\documentclass[a4paper, 12pt]{article}

\usepackage{amsfonts}
\usepackage{amsmath}
\usepackage{amssymb}
\usepackage{amsthm}
\usepackage{mathtools}
\usepackage{enumerate}
\usepackage{subdepth}
\usepackage[margin=1in]{geometry}

\newtheorem*{theorem*}{Theorem}
\newtheorem{theorem}{Theorem}
\newtheorem{lemma}[theorem]{Lemma}
\newtheorem{proposition}[theorem]{Proposition}

\newtheorem{remark}{Remark}
\newenvironment{definition}[1][Definition]{\begin{trivlist}
\item[\hskip \labelsep {\bfseries #1}]}{\end{trivlist}}

\def\V{\mathcal{V}}
\def\B{\mathcal{B}}

\def\J{\mathcal{J}}
\def\K{\mathcal{K}}

\def\P{\mathcal{P}}
\def\X{\mathfrak{X}}
\def\H{\mathcal{H}}
\def\K{\mathcal{K}}
\def\U{\mathcal{U}}
\def\E{\mathcal{E}}

\mathtoolsset{showonlyrefs}

\title{Generic Properties of Geodesic Flows on Analytic Hypersurfaces of Euclidean Space}
\author{
  Andrew Clarke\\
}
\date{
Universitat Polit\`ecnica de Catalunya \\
andrew.michael.clarke@upc.edu \\[2ex]
\today
}

\begin{document}
\maketitle
\begin{abstract} \noindent
Consider the geodesic flow on a real-analytic closed hypersurface $M$ of $\mathbb{R}^n$, equipped with the induced metric. How commonly can we expect such flows to have a transverse homoclinic orbit? In this paper, we give the following two partial answers to this question: 
\begin{itemize}
\item
If $M$ is a real-analytic closed hypersurface in $\mathbb{R}^n$ (with $n \geq 3$) on which the geodesic flow with respect to the induced metric has a nonhyperbolic periodic orbit, then $C^{\omega}$-generically the geodesic flow on $M$ with respect to the induced metric has a hyperbolic periodic orbit with a transverse homoclinic orbit; and
\item
There is a $C^{\omega}$-open and dense set of real-analytic, closed, and strictly convex surfaces $M$ in $\mathbb{R}^3$ on which the geodesic flow with respect to the induced metric has a hyperbolic periodic orbit with a transverse homoclinic orbit.
\end{itemize}
These are among the first perturbation-theoretic results for real-analytic geodesic flows.
\end{abstract}

\section{Introduction} \label{sec_introduction}

It is well known that the geodesic flow on an ellipsoid is an integrable Hamiltonian system \cite{jacobi1884vorlesungen} (see also \cite{moser1980geometry, moser1980various} and \cite{tabachnikov2002ellipsoids} for modern proofs). It is believed that $n$-dimensional ellipsoids and surfaces of revolution are among very few examples of integrable geodesic flows on closed hypersurfaces of Euclidean space. On the other hand, the presence of a nontrivial hyperbolic basic set implies the existence of chaotic motions. Standard results imply that if the system has a hyperbolic periodic orbit and a transverse homoclinic, then it has a nontrivial hyperbolic basic set \cite{poincare1899methodes, shilnikov1965case, smale1967differentiable}(see also Theorem 6.5.5 of \cite{katok1995introduction}). The existence of such a set is a $C^2$-open property of the Hamiltonian function, and in the current article the Hamiltonian depends on the curvature of the hypersurface (i.e. the second derivative). Therefore the existence of such a set is a $C^4$-open property (and therefore $C^r$-open for $r=5,6\dots, \infty, \omega$) of the hypersurface. A long-term goal is to prove that this property is also \emph{dense} in the $C^{\omega}$-topology. The geodesic flow on manifolds of negative curvature is Anosov \cite{anosov1967geodesic}, but less is known in the convex case. This paper establishes two results in this direction. Firstly, there is a residual set of closed real-analytic hypersurfaces $M$ in $\mathbb{R}^n$ (where $n \geq 3$) such that if $M$ has a nonhyperbolic closed geodesic, then the geodesic flow with respect to the Euclidean metric has a nontrivial hyperbolic basic set. Secondly, there is a $C^{\omega}$ open and dense set of real-analytic, closed, and strictly convex surfaces in $\mathbb{R}^3$ on which the geodesic flow with respect to the Euclidean metric has a nontrivial hyperbolic basic set.

\subsection{Main Results}

Let $d \geq 1$, and denote by $\V$ the set of all real-analytic functions $Q: \mathbb{R}^{d+2} \to \mathbb{R}$ such that the set
\begin{equation} \label{eq_manifoldequation}
M = M(Q) = \{ x \in \mathbb{R}^{d+2} : Q(x) = 0 \}
\end{equation}
is a closed hypersurface of $\mathbb{R}^{d+2}$. The geodesic flow $\phi^t$ takes a point $x \in M$ and a tangent vector $u \in T_x M$ and follows the unique geodesic through $x$ in the direction $u$ at a constant speed $\| u \|$. The energy $\frac{u^2}{2}$ is preserved.

Closed geodesics on $M$ correspond to periodic orbits of the geodesic flow. A method of construction of closed geodesics on any Riemannian manifold was proposed by Birkhoff \cite{birkhoff1917dynamical, birkhoff1927dynamical} (see also \cite{croke1988area}; see \cite{colding2011course} for a modern exposition). This implies that the geodesic flow on $M$ has a periodic orbit $\gamma$. Consider a transverse section to $\gamma$ in $TM$, and the corresponding Poincar\'e map of the periodic orbit $\gamma$ in the energy level $\| u \| = c$. The periodic orbit $\gamma$ is said to be:
\begin{itemize}
\item
\emph{Parabolic} if 1 is an eigenvalue of the linearisation of the Poincar\'e map;
\item
\emph{Degenerate} if the linearised Poincar\'e map has an eigenvalue equal to a root of unity;
\item
\emph{Hyperbolic} if the linearised Poincar\'e map has no eigenvalue of absolute value 1;
\item
\emph{Elliptic} if it is nondegenerate and nonhyperbolic; and
\item
\emph{$q$-elliptic} if the linearised Poincar\'e map has exactly $2q$ eigenvalues of absolute value 1.
\end{itemize}
Since the geodesic flow is a Hamiltonian system, the eigenvalues of the linearised Poincar\'e map come in reciprocal pairs. As the dynamics of the geodesic flow is the same on every energy level (see Section \ref{sec_geometryandperturbations}), there is a periodic orbit $\gamma$ in each energy level, and moreover the above classification is independent of the level set in consideration. Furthermore, the classification is independent of the choice of Poincar\'e map \cite{klingenberg1976lectures}. 

Let $\gamma$ be a hyperbolic periodic orbit of the geodesic flow $\phi^t$. Let $\theta \in \gamma$. Recall that the strong stable and strong unstable manifolds $W^{s,u}(\theta)$ are defined as
\begin{align}
W^s(\theta) &= \{ \bar{\theta} \in TM : \| \phi^t(\bar{\theta}) - \phi^t (\theta) \| \to 0 \text{ as } t \to+ \infty \}, \\
W^u(\theta) &= \{ \bar{\theta} \in TM : \| \phi^t(\bar{\theta}) - \phi^t (\theta) \| \to 0 \text{ as } t \to- \infty \}.
\end{align}
The stable and unstable manifolds $W^{s,u}(\gamma)$ of the hyperbolic periodic orbit $\gamma$ are 
\begin{equation}
W^s (\gamma) = \bigcup_{\theta \in \gamma} W^s (\theta), \quad W^u(\gamma) = \bigcup_{\theta \in \gamma} W^u(\theta),
\end{equation}
and we have
\begin{equation}
\dim W^s(\gamma) = \dim W^u(\gamma) = d+1.
\end{equation}
Now suppose $\eta$ is another hyperbolic periodic orbit. A point $\bar{\theta} \in TM$ is \emph{heteroclinic} if there are $\theta_1 \in \gamma, \theta_2 \in \eta$ such that
\begin{equation}
\| \phi^t (\bar{\theta}) - \phi^t (\theta_1) \| \to 0 \text{ as }t \to + \infty \quad \text{and} \quad \| \phi^t (\bar{\theta}) - \phi^t (\theta_2) \| \to 0 \text{ as }t \to - \infty.
\end{equation}
Clearly then $\bar{\theta} \in W^s(\gamma) \cap W^u(\eta)$. Suppose the orbits $\gamma, \eta$ lie in the unit tangent bundle $T^1M$ consisting of tangent vectors with unit length. Then $\bar{\theta}$ is a \emph{transverse heteroclinic} point if $\bar{\theta} \in W^s(\gamma) \cap W^u(\eta)$ and
\begin{equation}
T_{\bar{\theta}} W^s(\gamma) + T_{\bar{\theta}} W^u(\eta) = T_{\bar{\theta}} T^1M.
\end{equation}
If $\gamma = \eta$ then $\bar{\theta}$ is a \emph{homoclinic} point. If a homoclinic or heteroclinic point is transverse, then so is every point in its orbit.

Define the real-analytic topology on $\V$ as follows. Let $K \subset \mathbb{R}^{d+2}$ be a compact set, and let $\hat{K}$ be a compact complex neighbourhood of $K$. If $Q_1, Q_2 \in \V$, by definition they admit holomorphic extensions $\hat{Q}_1, \hat{Q}_2$ on $\hat{K}$. We say that $Q_1, Q_2$ are close on the compact set $K$ in the real-analytic topology if $\hat{Q}_1, \hat{Q}_2$ are uniformly close on $\hat{K}$. Recall that a subset of $\V$ is \emph{residual} if it is a countable intersection of open dense sets.

\begin{theorem} \label{theorem_homoclinicnearelliptic}
There is a residual set $\V_0 \subset \V$ such that for all $Q \in \V_0$, if the geodesic flow on $M(Q)$ with respect to the Euclidean metric has a nonhyperbolic periodic orbit, then it has a hyperbolic periodic orbit with a transverse homoclinic.
\end{theorem}

Denote by $\V^c$ the subset of $\V$ consisting of functions $Q$ for which $M = M(Q)$ is strictly convex. In the case of real-analytic, strictly convex surfaces in $\mathbb{R}^3$ we obtain the following general result. 

\begin{theorem} \label{theorem_3dgenerichyperbolicset}
If $d=1$ so that for each $Q \in \V^c$, the set $M = M(Q)$ is a real-analytic, closed, and strictly convex surface in $\mathbb{R}^3$, then there is a $C^{\omega}$ open and dense set $\V^* \subset \V^c$ such that for every $Q \in \V^*$, the geodesic flow with respect to the Euclidean metric on the manifold $M(Q)$ has a hyperbolic periodic orbit with a transverse homoclinic.
\end{theorem}

\begin{remark}
It is possible that Theorem \ref{theorem_3dgenerichyperbolicset} is also true in the nonconvex case (i.e. for an open and dense set in $\V$). The proof relies on the existence of a global surface of section. In the strictly convex case, the existence of such a surface was shown by Birkhoff \cite{birkhoff1927dynamical}. In the case of 3-dimensional Reeb flows, conditions were given by Hofer, Wysocki, and Zehnder which, if satisfied, guarantee the existence of a global surface of section \cite{hofer2002pseudoholomorphic, hofer2003finite}. The restriction of the geodesic flow to the unit tangent bundle is a 3-dimensional Reeb flow. One would need to check that the conditions of those papers are $C^{\omega}$-generically satisified by geodesic flows on real-analytic closed surfaces in $\mathbb{R}^3$, before applying the same reasoning as in Section \ref{section_3dgenericityofhyperbolicsets}.
\end{remark}

We point out that Theorems \ref{theorem_homoclinicnearelliptic} and \ref{theorem_3dgenerichyperbolicset} are well-known in the case of $C^{\infty}$ perturbations of the metric (as opposed to the case of $C^{\omega}$ perturbations of the hypersurface considered here). A comparison of Theorems \ref{theorem_homoclinicnearelliptic} and \ref{theorem_3dgenerichyperbolicset} to existing results is made in Section \ref{sec_comparisonwithexistingresults}, and a discussion of the additional challenges presented by the setting of analytic perturbations of the hypersurface can be found in Section \ref{sec_novelcontributions}.

\subsection{Main Ideas of the Proof}

The main tool (Theorem \ref{theorem_kupkasmale}) used in the proof of both Theorem \ref{theorem_homoclinicnearelliptic} and Theorem \ref{theorem_3dgenerichyperbolicset} is an analogue of the Kupka-Smale theorem for geodesic flows on analytic hypersurfaces $M=M(Q)$ with $Q \in \V$. It tells us that, for generic $Q\in \V$, the $k$-jet of the Poincar\'e map of every closed geodesic on $M(Q)$ is in general position, and all homoclinic and heteroclinic connections between hyperbolic closed geodesics on $M(Q)$ are transverse. Before stating the theorem, we introduce some salient features of the set of $k$-jets. 

For $k \in \mathbb{N}_0$, let $J_s^k (d)$ denote the set of $k$-jets of symplectic autormorphisms of $\mathbb{R}^{2d}$ that fix the origin. If $f$ is a symplectic automorphism of $\mathbb{R}^{2d}$ with a fixed point $x$, we let $J^k_x f$ denote the $k$-jet of $f$ at $x$. For two such symplectic automorphisms $f,g$, define the product 
\begin{equation}
J^k_x f \cdot J^k_x g = J^k_x (g \circ f).
\end{equation}
Notice that the Poincar\'e map of a geodesic on $M$ is such a symplectic automorphism of a neighbourhood of the origin after making a suitable coordinate transformation (e.g. Fermi coordinates - see Section \ref{sec_geometryandperturbations}). Moreover, the geodesic need not be closed: one can simply consider the Poincar\'e map between two transverse sections along a geodesic. A set $\J \subset J_s^k (d)$ is said to be \emph{invariant}\footnote{The reason we require invariance is that the property of the $k$-jet of a Poincar\'e map belonging to an invariant set $\J \subseteq J^k_s(d)$ is independent of the choice of coordinates and Poincar\'e section.} if $\sigma \J \sigma^{-1} = \J$ for all $\sigma \in J_s^k (d)$.

\begin{theorem} \label{theorem_kupkasmale}
Let $k \in \mathbb{N}$ and let $\J \subset J^k_s(d)$ be open, dense, and invariant. There is a residual set $\K \subset \V$ such that if $Q \in \K$ and $M = M(Q)$, then:
\begin{enumerate}[(i)]
\item \label{item_kupkasmale1}
The $k$-jet of the Poincar\'e map of every closed geodesic on $M$ lies in $\J$; and
\item \label{item_kupkasmale2}
Any homoclinic or heteroclinic connections between hyperbolic closed geodesics are transverse.
\end{enumerate}
\end{theorem}

We say that a hypersurface $M=M(Q)$ is of \emph{Kupka-Smale type} if $Q \in \K$. The deduction of Theorem \ref{theorem_homoclinicnearelliptic} from Theorem \ref{theorem_kupkasmale} is analogous to the deduction of Theorem C of \cite{contreras2010geodesic} from Theorem 2.5 of \cite{contreras2002genericity} (itself based on ideas of Zehnder \cite{zehnder1973homoclinic}), and so it is not repeated in this paper. The deduction of Theorem \ref{theorem_3dgenerichyperbolicset} from Theorem \ref{theorem_kupkasmale} is presented in Section \ref{section_3dgenericityofhyperbolicsets}. 

Theorem \ref{theorem_kupkasmale} itself is proved in Section \ref{section_kupkasmale}, and its proof has three main ingredients. The first of these, Theorem \ref{theorem_ktlocal}, allows us to bring the $k$-jet of the Poincar\'e map of a \emph{single} geodesic segment into general position by arbitrarily small analytic perturbations of the hypersurface. 

\begin{theorem} \label{theorem_ktlocal}
Let $Q \in \V$, and let $M=M(Q)$. Let $k \in \mathbb{N}$ and let $\J \subset J^k_s(d)$ be open, dense, and invariant. Let $\gamma : [0,1] \to M$ denote a nonconstant geodesic segment such that the normal curvature along $\gamma$ is not identically zero. Assume moreover that $\gamma (0)$ is not a point of self-intersection of $\gamma$. Let $l : [0,1] \to TM$ be the corresponding orbit segment $l=(\gamma, \gamma')$ of the geodesic flow. Let $\Sigma_0, \Sigma_1$ be transverse sections to $l$ at $l(0), l(1)$ respectively. Let $P_Q$ denote the Poincar\'e map from $\Sigma_0$ to $\Sigma_1$. Then we can find $\tilde{Q} \in \V$ arbitrarily close to $Q$ such that the perturbed Poincar\'e map $P_{\tilde{Q}}$ corresponding to the perturbed orbit segment $\tilde{l} : [0,1] \to TM(\tilde{Q})$ satisfies $J^k_{\tilde{l}(0)} P_{\tilde{Q}} \in \J$.
\end{theorem}

\begin{remark} \label{remark_notsamemanifold}
The orbits $l, \tilde{l}$ do not lie on the same manifold, so it is not immediately clear that we can compare them. However we can fix some reference manifold $M_0$ such that the hypersurface $M$ in consideration admits a $C^{\omega}$-diffeomorphism $G:M_0 \to M$, so we can pull back the Euclidean metric on $M$ to $M_0$ via $G$ to get a Riemannian metric on the hypersurface $M_0$. Therefore we can consider $l, \tilde{l}$ as orbit segments of geodesic flows corresponding to different metrics on $M_0$. See Section \ref{sec_hypersurfaceperturbations}.
\end{remark}

\begin{remark}
The assumption that the normal curvature along $\gamma$ is not identically zero is necessary because the result of a perturbation of the hypersurface on the metric in Fermi coordinates is equal to some function times a matrix related to the curvature matrix (see Section \ref{sec_geometryandperturbations} for definitions and details). In particular, if the curvature is zero at a point, then so is the perturbation of the metric. This assumption is not particularly restrictive, as closed geodesics necessarily have this property, and therefore so too do homoclinic and heteroclinic geodesics. Moreover, given a geodesic segment with normal curvature identically zero, it is expected that a generic $C^{\omega}$-small perturbation of the hypersurface would curve this geodesic segment. 
\end{remark}

We prove Theorem \ref{theorem_ktlocal} in Section \ref{section_klingenbergtakens}. The second ingredient (Theorem \ref{theorem_transversalityofconnections}, which is stated and proved in Section \ref{section_kupkasmale}) of the proof of Theorem \ref{theorem_kupkasmale} allows us to make homoclinic and heteroclinic connections transverse by arbitrarily small analytic perturbations of the hypersurface. Note that Theorem \ref{theorem_kupkasmale} is not immediate from Theorems \ref{theorem_ktlocal} and \ref{theorem_transversalityofconnections}, for the following reason. We would like to proceed with the proof of Theorem \ref{theorem_kupkasmale} as follows: beginning with a hypersurface $M$, we show, by applying Theorems \ref{theorem_ktlocal} and \ref{theorem_transversalityofconnections} to each closed geodesic on $M$ of length at most $N$ (for some $N>0$), that we can obtain a hypersurface on which every closed geodesic of length at most $N$ satisfies properties $(i)$ and $(ii)$ of Theorem \ref{theorem_kupkasmale}. Denote by $\K(N)$ the set of such hypersurfaces. Therefore $\K(N)$ is dense; it is also open. Thus intersecting $\K(N)$ over all $N \in \mathbb{N}$ yields the residual set $\K$. The problem with this argument is that an arbitrary hypersurface $M$ may contain \emph{infinitely many} closed geodesics of length at most $N$  for finite positive $N$, and it does not make sense to apply the perturbation theorems infinitely many times. The following theorem is the last ingredient in the proof of Theorem \ref{theorem_kupkasmale}, as it provides us with a residual set $\B$ such that, for every $Q \in \B$, there are only finitely many closed geodesics on $M(Q)$ of length at most $N$ for any $N>0$. The proof of Theorem \ref{theorem_bumpymetric} is contained in Section \ref{section_bumpymetric}.

\begin{theorem} \label{theorem_bumpymetric}
There is a residual set $\B \subset \V$ such that if $Q \in \B$ then every closed geodesic on $M(Q)$ is nondegenerate.
\end{theorem}

\subsection{Comparison to Existing Results}\label{sec_comparisonwithexistingresults}

In the following discussion, the setting of (at most $C^{\infty}$) perturbations of Riemannian metrics on a fixed closed manifold is referred to as the \emph{classical setting}.

In the classical setting, it was proved by Contreras that there is a $C^2$-open and dense set of $C^{\infty}$ Riemannian metrics on an arbitrary closed manifold $M$ for which the geodesic flow has a nontrivial hyperbolic basic set \cite{contreras2010geodesic}. His result hinges on an application of Ma\~n\'e's theory of dominated splittings \cite{mane1983oseledec}, and an analogue of Franks' Lemma for geodesic flows, which only works for $C^2$-small perturbations of the metric. Knieper and Weiss proved that, whenever there exists a global surface of section, by an arbitrarily $C^{\infty}$ small perturbation of the metric, a geodesic flow with a nontrivial hyperbolic basic set can be obtained \cite{knieper2002c}. Combining this with results of Hofer, Wysocki, and Zehnder \cite{hofer2002pseudoholomorphic, hofer2003finite} and an argument of Knieper and Weiss \cite{knieper2002c} implies $C^{\infty}$-generic existence of nontrivial hyperbolic basic sets on surfaces in $\mathbb{R}^3$ in the classical setting.

Theorem \ref{theorem_kupkasmale} is an analogue of the Kupka-Smale theorem for geodesic flows on real-analytic, closed, and strictly convex hypersurfaces of Euclidean space. Part \eqref{item_kupkasmale1} of Theorem \ref{theorem_kupkasmale} follows from Theorems \ref{theorem_ktlocal} and \ref{theorem_bumpymetric}, as pointed out by Anosov \cite{anosov1983generic}. Part \eqref{item_kupkasmale2} was proved in the classical setting by Contreras and Paternain \cite{contreras2002genericity}.

Theorem \ref{theorem_ktlocal} was originally proved for $k$-jets of $C^{k+1}$-smooth Poincar\'e maps arising in the classical setting in the $C^{k+1}$-topology for any $k \in \mathbb{N}$ by Klingenberg and Takens \cite{klingenberg1972generic}. For hypersurfaces with the Euclidean metric, it was first shown by Stojanov and Takens in \cite{stojanov1993generic}, but this result did not include the analytic case. An analogue for the case of Ma\~n\'e generic Hamiltonians was proved for $k=1$ in \cite{rifford2012generic} and for $k \geq 2$ in \cite{carballo2013jets}.

Theorem \ref{theorem_bumpymetric} was originally stated (without proof) by Abraham in the classical setting \cite{abraham1970bumpy} for $C^r$-smooth metrics where $5 \leq r < \infty$. Some time later, the first proof was produced by Anosov \cite{anosov1983generic}, which worked for $2 \leq r \leq \infty$. A $C^{\infty}$ version of the theorem for hypersurfaces with the Euclidean metric was proved by Stojanov \cite{stojanov1990bumpy}. The classical result is typically referred to as the bumpy metric theorem, but we resist that nomenclature in this instance to avoid any potential confusion with essential use of bump functions, since they fail to be real-analytic.

It was shown by Contreras that Theorem \ref{theorem_homoclinicnearelliptic} is equivalent to Theorem \ref{theorem_kupkasmale} in the classical setting \cite{contreras2010geodesic}. As his proof  of this equivalence applies directly, once Theorem \ref{theorem_kupkasmale} is proved, to the case of geodesic flows on real-analytic closed hypersurfaces of Euclidean space, we do not include it here. The complicated structure of orbits near elliptic periodic points was proved in the case of real-analytic symplectic diffeomorphisms of the plane with an elliptic fixed point by Zehnder in the 70's \cite{zehnder1973homoclinic}, and the proof in the case of geodesic flows is based on the ideas in that paper. The general idea is that the restriction of the Poincar\'e map of an elliptic closed geodesic to its centre manifold can be put in a Birkhoff normal form. Then, using techniques developed by Arnaud and Herman \cite{arnaud1992type}, it can be shown that such a map generically has a 1-elliptic periodic orbit. The restriction of the Poincar\'e map corresponding to the 1-elliptic periodic orbit to its 2-dimensional centre manifold is a twist map of the annulus, of Kupka-Smale type, and therefore has a transverse homoclinic orbit \cite{le1991proprietes}.

\subsection{Novel Contributions of This Paper}\label{sec_novelcontributions}

The reason that the proofs of these theorems from the classical setting do not apply directly to the present case is that the class of perturbations allowed here is more restrictive, for two reasons. The first reason is as follows. The Nash embedding theorem guarantees that any Riemannian manifold can be isometrically embedded in some Euclidean space $\mathbb{R}^N$, but small perturbations of the metric do not necessarily preserve $N$. In the next section we compute the effect of a perturbation of the hypersurface on the metric, and we will see that this results only in a small subset of all possible perturbations of the metric (see Lemmas \ref{lemma_embeddingexistence} and \ref{lemma_perturbationfermicoords}). It follows that many perturbations of the metric are not allowed in the current setting.

The second reason that our perturbations are more restrictive than the usual case is that the proofs presented here apply in the real-analytic topology (as well as weaker topologies). The real-analytic topology is very restrictive, as it does not allow the explicit use of bump functions when making perturbations. Therefore any perturbation of the hypersurface affects dynamics globally. This means there are phenomena that are $C^{\infty}$ generic, but cannot be obtained by a $C^{\omega}$ small perturbation, and so it is not clear a priori that the phenomena considered in this paper are $C^{\omega}$ generic. To overcome this difficulty, a method introduced by Broer and Tangerman \cite{broer1986differentiable} is used: determine open conditions (in a weaker topology, e.g. $C^4$) to be satisfied by a family of perturbations to obtain the desired effect for arbitrarily small values of the parameter; show that these conditions are satisfied by a locally-supported family of perturbations; and approximate the family of perturbed systems sufficiently well by a real-analytic family. Since the conditions are open in a weaker topology, they are satisfied by the real-analytic family since the approximation can be made arbitrarily well in $C^r$. In fact, this method can be used together with the proof of Anosov to obtain a real-analytic version of the bumpy metric theorem, and with the proof of Contreras and Paternain to obtain a real-analytic version of the Kupka-Smale theorem in the classical setting. This technique was also used in \cite{clarke2019arnold, gelfreich2017arnold, gonchenko2007homoclinic}. An example of a perturbation scheme that \emph{cannot} be adapted to the $C^{\omega}$ topology using this method can be found in \cite{cheng2004existence}. 

\subsection{Applications}

    The primary motivation for studying this problem is that the hypersurfaces obtained in Theorems \ref{theorem_homoclinicnearelliptic} and \ref{theorem_3dgenerichyperbolicset} satisfy the assumptions of the main theorem in \cite{clarke2019arnold}. In that paper it was shown that in the subset of $\V^c$ consisting of functions $Q$ for which the geodesic flow on $M(Q)$ has a hyperbolic periodic orbit and a transverse homoclinic, generically a form of Arnold diffusion occurs for billiard dynamics inside the hypersurface. The results of this paper show that these assumptions are satisfied generically for surfaces in $\mathbb{R}^3$, and generically near a hypersurface with an elliptic closed geodesic; a future goal is to prove that these assumptions are generically satisfied when every closed geodesic is hyperbolic. Such result would prove Theorem \ref{theorem_3dgenerichyperbolicset} in the case when $d>1$, and thus show that the Arnold diffusion phenomenon of \cite{clarke2019arnold} is generic in the set $\V^c$ in any dimension greater than or equal to 3. As a consequence of the results in the present paper, the genericity obtained in \cite{clarke2019arnold} in dimension 3 therefore no longer requires assumptions about the geodesic flow; moreover in higher dimension, the present results show that the genericity in \cite{clarke2019arnold} is obtained in a \emph{nonempty set}.

It is expected that the results of this paper will be used in further applications, in particular in proofs of instability of Hamiltonian systems. Indeed, the setting of perturbations of geodesic flows has been used, for example, in \cite{delshams2006orbits,gidea2017perturbations}, where it is assumed that there is a transverse homoclinic geodesic. Due to the aforementioned difficulties, there are not currently many proofs regarding Hamiltonian instability in the analytic topology, but the results of this paper provide a foundation from which such proofs could begin. 

The structure of this paper is as follows. In Section \ref{sec_geometryandperturbations}, we introduce some basic ideas regarding the geometry of the hypersurface, and show how a (local) perturbation of the manifold affects the geodesic flow. In Section \ref{section_bumpymetric}, Theorem \ref{theorem_bumpymetric} is proved. In Section \ref{section_klingenbergtakens}, Theorem \ref{theorem_ktlocal} is proved, and Theorem \ref{theorem_kupkasmale} is proved in Section \ref{section_kupkasmale}, along with a theorem that breaks homoclinic or heteroclinic connections by arbitrarily small real-analytic perturbations of the hypersurface (see \cite{contreras2002genericity, donnay1995transverse, petroll1996existenz} for proofs of similar results in the classical setting). In Section \ref{section_3dgenericityofhyperbolicsets}, it is shown how to combine results of this paper with the argument of Knieper and Weiss to prove Theorem \ref{theorem_3dgenerichyperbolicset}.

\vspace{5mm}
\noindent
\textit{Acknowledgements.}$\quad$The author would like to thank Dmitry Turaev for proposing the problem; Gonzalo Contreras, Marcel Guardia, Andr\'{e} Neves, and Guillermo Olic\'on M\'endez for useful conversations; and Luchezar Stoyanov for providing references \cite{stojanov1990bumpy, stojanov1993generic}. This project has received funding from the European Research Council (ERC) under the European Union’s Horizon 2020 research and innovation programme (Grant Agreement No 757802).

\section{Geometry and Perturbations} \label{sec_geometryandperturbations}
\subsection{Geometry of the Domain}

Let $Q \in \V$, and let $M = M(Q)$ be defined as in \eqref{eq_manifoldequation}. Then for $x \in M$, the unit normal vector
\begin{equation} \label{eq_unitnormal}
n(x) = - \frac{ \nabla Q (x)}{\| \nabla Q (x) \|}
\end{equation}
is inward-pointing\footnote{The minus sign is a convention that leads to the standard definitions of convexity and concavity.}. The curvature matrix $C(x)$ is the matrix of second partial derivatives of $Q$ divided by the norm of the gradient:
\begin{equation} \label{eq_curvaturematrix}
C(x) = \| \nabla Q(x) \|^{-1} \left( \frac{\partial^2 Q}{\partial x_i \partial x_j} (x) \right)_{i,j=1, \ldots, d+2}
\end{equation}
Let $u \in T_x M$. The shape operator $S(x): T_x M \to T_x M$ is defined as
\begin{equation} \label{eq_shapeoperator}
S(x)u = -Dn(x)u = C(x)u - \langle C(x)u, n(x) \rangle n(x).
\end{equation}
This enables a definition of the normal curvature at $x$ in the direction $u$ via
\begin{equation} \label{eq_normalcurvature}
\kappa(x,u) = \langle S(x)u,u \rangle = \langle C(x)u, u \rangle.
\end{equation}
Strict convexity means that for all $x \in M$ and $0 \neq u \in T_x M$ we have
\begin{equation} \label{eq_strictconvexity}
\kappa (x,u) >0.
\end{equation}
Let 
\begin{align}
\pi : TM & \longrightarrow M \\
T_x M \ni  u & \longmapsto x
\end{align}
denote the canonical projection along fibres of the tangent bundle. We use the notation $(x,u) \in TM$ to mean $x \in M$ and $u \in T_x M$, so that $\pi (x,u) = x$.

\subsection{The Geodesic Flow}

The geodesic flow on $M$ takes a tangent vector $(x,u) \in T M$, and follows the uniquely defined geodesic starting at $x$ in the direction of $u$ at a constant speed of $\| u \|$. Typically, it is introduced via the Hamiltonian function 
\begin{equation} \label{eq_metrichamiltonian}
H(x,u) = \frac{1}{2} g(x) (u,u)
\end{equation}
where $g$ is a Riemannian metric, and $(x,u)$ are intrinsic coordinates on $TM$. In our case we use the induced metric and the coordinates of the ambient Euclidean space $\mathbb{R}^{d+2}$, so a different formulation is required.

Consider a smooth curve $\gamma : [a,b] \to M$. The tangential component $\gamma''(t)^{\top}$ of its acceleration is given by
\begin{equation} \label{eq_geodesicacceleration}
\gamma''(t)^{\top} = \gamma''(t) - \langle \gamma''(t), n (\gamma(t)) \rangle n(\gamma (t)).
\end{equation}
Since $\gamma'(t) \in T_{\gamma (t)} M$ and $n(\gamma (t))$ is normal to $M$ at $\gamma (t)$ for all $t \in [a,b]$ we have
\begin{equation} \label{eq_secondderivativecurvaturerelation}
0 = \frac{d}{dt} \langle \gamma'(t), n(\gamma (t)) \rangle = \langle \gamma''(t), n(\gamma(t)) \rangle - \kappa (\gamma (t), \gamma'(t)).
\end{equation}
The curve $\gamma$ is a geodesic if and only if $\gamma''(t)^{\top}=0$, so from \eqref{eq_geodesicacceleration} and \eqref{eq_secondderivativecurvaturerelation} we see that $\gamma$ is a geodesic if and only if
\begin{equation}
\gamma^{\prime \prime}(t) = \kappa (\gamma (t), \gamma^{\prime}(t)) \, n(\gamma (t)).
\end{equation}
It follows that the geodesic flow $\phi^t : T M \to T M$ is defined by 
\begin{equation} \label{eq_geodesicflowdef}
\left. \frac{d}{dt} \right|_{t=0} \phi^t (x,u) = X(x,u),
\end{equation}
where the vector field $X(x,u) = (\dot{x}, \dot{u})$ is given by
\begin{equation} \label{eq_geodesicfloweom}
	\begin{dcases}
		\dot{x} = u \\
		\dot{u} = \kappa (x,u) \, n (x).
	\end{dcases}
\end{equation}
Consider the function $H : T M \to \mathbb{R}$ given by
\begin{equation} \label{eq_geodesicflowhamiltonian}
H(x,u) = \frac{u^2}{2} + \kappa (x,u) \frac{Q(x)}{\| \nabla Q(x) \|}.
\end{equation}
It is not hard to see that $X = \Omega_{d+2} \nabla H$ where 
\begin{equation} \label{eq_standardsymplecticmatrix}
\Omega_{d+2} = \left(
\begin{array}{cc}
0 & I_{d+2} \\
-I_{d+2} & 0
\end{array} \right)
\end{equation}
is the standard symplectic matrix (and $I_{d+2}$ is the $(d+2)$-dimensional identity matrix). This is equivalent to the statement that the geodesic flow is a Hamiltonian system with symplectic form
\begin{equation} \label{eq_symplecticform}
\omega = dx \wedge du.
\end{equation}
In particular, the geodesic flow is the Hamiltonian flow associated with the Hamiltonian function $H$.

Notice that the second term of \eqref{eq_geodesicflowhamiltonian} vanishes identically on $T M$, and so the energy $\frac{u^2}{2}$ is conserved. Since the Hamiltonian $H$ is homogeneous of second order in $u$, the dynamics of the geodesic flow is the same on every energy level.

\subsection{Perturbations of the Hypersurface}\label{sec_hypersurfaceperturbations}

Let $Q \in \V$ and $M = M(Q)$. Consider the set $\E_Q$ of $C^{\omega}$ embeddings $G: M \hookrightarrow \mathbb{R}^{d+2}$, equipped with the real-analytic topology. For any $G \in \E_Q$, the set $G(M)$ is a $C^{\omega}$ hypersurface of $\mathbb{R}^{d+2}$. We analyse perturbations of $M$ by considering embeddings $G \in \E_Q$ that are close to the identity. In practice, it is inconvenient to deal with analytic embeddings, as they are supported globally on the hypersurface. Instead, we show how a local perturbation of the function $Q$ gives rise to a $C^{\infty}$ family of embeddings that differ from the identity only in a neighbourhood of a point on $M$. We will then use the technique of Broer and Tangerman to pass from $C^{\infty}$ families of hypersurfaces to $C^{\omega}$ perturbations of the original hypersurface $M$.

Let $x \in M$ and suppose we make a local perturbation
\begin{equation} \label{eq_perturbation}
Q \longrightarrow Q + \epsilon \psi
\end{equation}
where $\epsilon$ is small and $\psi$ is a $C^{\infty}$ function supported near $x$. This gives rise to a $C^{\infty}$ family of hypersurfaces $M_{\epsilon} = M(Q + \epsilon \psi)$ where $M_0 \equiv M$. 
\begin{lemma}\label{lemma_embeddingexistence}
There is a $C^{\infty}$ family of embeddings $G_{\epsilon}: M \hookrightarrow \mathbb{R}^{d+2}$ such that
\begin{equation}\label{eq_embeddingformula}
G_{\epsilon} (M) = M_{\epsilon}, \quad G_{\epsilon}(x) = x + \epsilon \, \psi (x) \, n(x) + O (\epsilon^2), 
\end{equation}
for all sufficiently small values of $\epsilon$, where $n(x)$ is the unit normal, defined by \eqref{eq_unitnormal}.
\end{lemma}
\begin{proof}
For all sufficiently small values of $\epsilon$, $M_{\epsilon}$ can be written as a graph over $M$. It follows that there is a $C^{\infty}$ family of functions $r_{\epsilon} : M \to \mathbb{R}$ for which the family of embeddings $G_{\epsilon} : M \hookrightarrow \mathbb{R}^{d+2}$ defined by
\begin{equation}
G_{\epsilon} (x) = x + \epsilon \, r_{\epsilon} (x) \, n(x)
\end{equation}
satisfies $G_{\epsilon}(M) = M_{\epsilon}$, for all sufficiently small values of $\epsilon$. Let us determine $r$, where $r_{\epsilon} = r + O(\epsilon)$. If $x \in M$, then $G_{\epsilon}(x) \in M_{\epsilon}$, so we must have
\begin{equation}
0=Q(G_{\epsilon}(x)) + \epsilon \psi (G_{\epsilon}(x)) = \epsilon \left[ r(x) \left\langle \nabla Q(x), n(x) \right\rangle + \psi (x) \right] + O(\epsilon^2)
\end{equation}
and so $r(x) = \| \nabla Q(x) \|^{-1} \psi (x)$ due to \eqref{eq_unitnormal}. Without loss of generality, we may assume that $\| \nabla Q(x) \| = 1$ for all $x \in M$. Indeed, if not, let $\tilde{Q} (x) = \| \nabla Q (x) \|^{-1} Q(x)$. Then $\tilde{Q} (x) =0$ if and only if $Q(x)=0$, so $M = M(\tilde{Q})$. Moreover, if $x \in M$ then
\begin{equation}
\nabla \tilde{Q}(x) = \| \nabla Q (x) \|^{-1} \nabla Q(x) + Q(x) \nabla \| \nabla Q (x) \|^{-1} = \| \nabla Q (x) \|^{-1} \nabla Q(x)
\end{equation}
since $Q(x)=0$ for $x \in M$. It follows that $\| \nabla \tilde{Q} (x) \| = 1$ for all $x \in M$, so we can replace $Q$ by $\tilde {Q}$. Equation \eqref{eq_embeddingformula} follows. 
\end{proof}
Using such embeddings $G_{\epsilon}$, we can consider the geodesic flow on $M_{\epsilon}$ with respect to the standard Euclidean metric $\left. \delta \right|_{M_{\epsilon}}$ as the geodesic flow on $M$ with respect to the Riemannian metric $g_{\epsilon} = G_{\epsilon}^* \left( \left. \delta \right|_{M_{\epsilon}} \right)$, thus allowing us to compare the perturbed and unperturbed flows. 

\subsection{Perturbations in Fermi Coordinates} \label{sec_fermicoordinates}

Let $Q \in \V$ and $M = M(Q)$. Recall that $\phi^t$ denotes the time-$t$ shift along orbits of the geodesic flow. Let $x \in M$, and define the \emph{exponential map}
\begin{align}
\exp_x : T_x M & \longrightarrow M \\
u & \longmapsto \phi^1 (x,u)
\end{align}
It is well-known that the exponential map is a diffeomorphism in a neighbourhood of $0 \in T_x M$.

Let $\gamma : [a,b] \to M$ be a nonconstant geodesic segment with no self-intersections, parametrised to have unit speed. Then we can choose an orthonormal basis $\gamma' (a), e_1, \ldots, e_d$ of $T_{\gamma (a)} M$. Moving this basis by parallel transport along the geodesic segment $\gamma$ gives an orthonormal basis $\gamma' (t), e_1 (t), \ldots, e_d (t)$ of $T_{\gamma (t)} M$ for each $t \in [a,b]$. It follows that there is a $\delta > 0$ such that if $B_{\delta}^d (0)$ denotes the ball of radius $\delta$ around the origin in $\mathbb{R}^d$ then the map
\begin{align}
\begin{split}\label{eq_fermichartmap}
h : [a,b] \times B_{\delta}^d (0) & \longrightarrow M \\
(y_0, y) = (y_0, y_1, \ldots, y_d) & \longmapsto \exp_{\gamma (y_0)} \sum_{j=1}^d y_j e_j (y_0)
\end{split}
\end{align}
defines a chart on $M$ (since $\gamma$ has no self-intersections). We can still use such `coordinates' if $\gamma$ has self-intersections, but they will not define a chart.

Let $(v_0, v) = (v_0, v_1, \ldots, v_d)$ denote the corresponding tangent coordinates, in the sense that $(v_0, v)$ corresponds to
\begin{equation}
\sum_{j=0}^d v_j \frac{\partial}{\partial y_j}
\end{equation}
in the tangent space. The geodesic $\gamma$ and its tangents, in these coordinates, are
\begin{equation} \label{eq_geodesicinfermicoords}
\gamma (t) = (t, 0) = (t, 0, \ldots, 0), \quad \gamma' (t) = (1,0) = (1, 0, \ldots, 0).
\end{equation} Let $g_{ij}$ denote the metric in these coordinates, and recall that $\delta_{ij}=1$ if $i=j$ and 0 otherwise. Then along the geodesic segment $\gamma$ we have
\begin{equation} \label{eq_fermiproperties}
g_{ij} = \delta_{ij}, \quad \frac{\partial}{\partial y_k} g_{ij} =0, \quad \frac{d}{dt} y_i = \delta_{0i}.
\end{equation}
In fact, the components of the metric in these coordinates are given by
\begin{equation}\label{eq_unperturbedcomponents}
g_{ij} (y_0,y) = \sum_{k,l=1}^{d+2} \frac{\partial x_k}{\partial y_i} \frac{\partial x_l}{ \partial y_j} \delta_{kl} = \left\langle \frac{\partial h}{\partial y_i} (y_0, y), \frac{\partial h}{\partial y_j} (y_0, y) \right\rangle
\end{equation}
where $h(y_0,y) = x$ is defined by \eqref{eq_fermichartmap}, where $\delta_{ij}$ are the components of the standard Euclidean metric, and where we have used the usual formula for the transformation of the metric under a change of coordinates \cite{stojanov1990bumpy, stojanov1993generic}. These coordinates were first introduced by Fermi and typically bear his name. See \cite{klingenberg1976lectures} for more details.

\begin{lemma}\label{lemma_perturbationfermicoords}
Suppose we make a perturbation $Q \to Q + \epsilon \psi$ as in \eqref{eq_perturbation}. Let $M_{\epsilon} = M (Q + \epsilon \psi)$, and denote by $G_{\epsilon} : M \hookrightarrow \mathbb{R}^{d+2}$ the corresponding embedding provided by Lemma \ref{lemma_embeddingexistence}. The variables $(y_0,y)$ define coordinates on a subset of $M_{\epsilon}$ via the map $G_{\epsilon} \circ h : [a,b] \times B_{\delta}^d (0) \to M_{\epsilon}$, and the pullback metric is
\begin{equation}
g_{\epsilon} = \left( G_{\epsilon} \circ h \right)^* \left( \left. \delta \right|_{M_{\epsilon}} \right) = g + \epsilon \bar{g} + O( \epsilon^2)
\end{equation}
where the components of $g$ are given by \eqref{eq_unperturbedcomponents}, and where the components of $\bar{g}$ are
\begin{equation}\label{eq_metricpertcomps}
\bar{g}_{ij} (y_0,y) = 2 \psi \left( h (y_0,y) \right) \tilde{C}_{ij} (y_0,y)
\end{equation}
where $\tilde{C}_{ij}$ are analytic functions of $(y_0,y)$ such that
\begin{equation}\label{eq_perturbationmatrix00}
\tilde{C}_{00} (y_0,0) = - \kappa \left( \gamma (y_0), \gamma' (y_0) \right), 
\end{equation}
where the normal curvature $\kappa$ is defined by \eqref{eq_normalcurvature}. 
\end{lemma}
\begin{proof}
We can write $G_{\epsilon} = \mathrm{Id} + \epsilon \Psi + O(\epsilon^2)$ where $\Psi (x) = \psi (x) n(x)$ by Lemma \ref{lemma_embeddingexistence}. By \eqref{eq_unperturbedcomponents}, the components of the pullback metric are
\begin{align}
\left( g_{\epsilon} \right)_{ij} (y_0,y) ={}& \left\langle \frac{\partial \left( G_{\epsilon} \circ h \right)}{\partial y_i} (y_0,y),  \frac{\partial \left( G_{\epsilon} \circ h \right)}{\partial y_j} (y_0,y) \right\rangle \\
={}& \left\langle \frac{\partial h}{\partial y_i} (y_0, y), \frac{\partial h}{\partial y_j} (y_0, y) \right\rangle + \epsilon \bigg[ \left\langle \frac{\partial \left( \Psi \circ h \right)}{\partial y_i} (y_0,y),  \frac{\partial h }{\partial y_j} (y_0,y) \right\rangle +  \\ 
& \qquad + \left\langle \frac{\partial h}{\partial y_i} (y_0,y),  \frac{\partial \left( \Psi \circ h \right)}{\partial y_j} (y_0,y) \right\rangle \bigg] + O (\epsilon^2) \\
={}& g_{ij} + \epsilon \bar{g}_{ij} + O( \epsilon^2)
\end{align}
where
\begin{equation}
\bar{g}_{ij} (y_0,y) = \left\langle \frac{\partial \left( \Psi \circ h \right)}{\partial y_i} (y_0,y),  \frac{\partial h }{\partial y_j} (y_0,y) \right\rangle + \left\langle \frac{\partial h}{\partial y_i} (y_0,y),  \frac{\partial \left( \Psi \circ h \right)}{\partial y_j} (y_0,y) \right\rangle. 
\end{equation}
By \eqref{eq_shapeoperator} we have
\begin{align}
\frac{\partial \left( \Psi \circ h \right)}{\partial y_i} (y_0,y) ={}& \left\langle \nabla \psi \left( h(y_0,y) \right), \frac{\partial h}{\partial y_i} (y_0,y) \right\rangle \, n \left(h(y_0,y) \right) - \\
& \qquad - \psi \left( h(y_0,y) \right) \, S \left( h(y_0,y) \right) \frac{\partial h}{\partial y_i} (y_0,y)
\end{align}
where $S$ is the shape operator. Since $\frac{\partial h}{\partial y_i} (y_0,y) \in T_{h(y_0,y)} M$ and $n \left( h(y_0,y) \right)$ is normal to $M$ at $h(y_0,y)$, it follows, using \eqref{eq_shapeoperator} again, that the components of $\bar{g}$ are given by \eqref{eq_metricpertcomps}, where
\begin{equation}
\tilde{C}_{ij} (y_0,y) = - \frac{1}{2} \bigg[ \left\langle C \left( h(y_0,y) \right) \frac{\partial h}{\partial y_i} (y_0,y), \frac{\partial h}{\partial y_j} (y_0,y) \right\rangle + \left\langle  \frac{\partial h}{\partial y_i} (y_0,y), C \left( h(y_0,y) \right) \frac{\partial h}{\partial y_j} (y_0,y) \right\rangle \bigg]
\end{equation}
with $C$ denoting the curvature matrix. Moreover, since $h (y_0) = \gamma (y_0)$, and since $y_0$ is the time along $\gamma$, equation \eqref{eq_perturbationmatrix00} follows. 
\end{proof}

\section{Generic Nondegeneracy of Closed Geodesics} \label{section_bumpymetric}

Let $Q_0 \in \V$ and let $M_0 = M(Q_0)$. Recall the definition of the space $\E_{Q_0}$ of embeddings from Section \ref{sec_hypersurfaceperturbations}. Let $G \in \E_{Q_0}$. Then there is $Q \in \V$ such that $G(M_0)=M$ where $M = M(Q)$. Recall that $G$ defines a metric $g = G^*\left( \left. \delta \right|_{M} \right)$ on $M_0$.

Let $x \in M_0$ and $u \in T_x M_0$. Recall the vertical subspace $T^V_u T M_0$ of $T_u T M_0$ is 
\begin{equation}
T^V_u T M_0 = T_u T_x M_0,
\end{equation}
and the horizontal subspace $T^H_u T M_0$ is the orthogonal complement of $T^V_u T M_0$ in $T_u T M_0$ with respect to $g$. Then we have 
\begin{equation}
T_u T M_0 = T^H_u TM_0 \oplus T^V_u T M_0 \simeq T_x M_0 \oplus T_x M_0.
\end{equation}
If $u_1, u_2 \in T_x M_0$, we let $(u_1; u_2) \in T_u T M_0$ denote the vector in the second tangent space defined by this correspondence.

Define
\begin{equation}
\begin{matrix}
\Phi : T M_0 \times \mathbb{R}_{> 0} \times \E_{Q_0} & \longrightarrow & T M_0 \\
\quad \qquad (\theta,t,G) & \longmapsto & \phi^t_G (\theta)
\end{matrix}
\end{equation}
where $\phi^t_G$ is the time-$t$ map of the geodesic flow on $M_0$ with respect to the pullback metric $G^* \delta$. 

The deduction of Theorem \ref{theorem_bumpymetric} from the following lemma is analogous to the deduction of Theorem 1 of \cite{anosov1983generic} from Lemma 2 of \cite{anosov1983generic}, and so we do not repeat the argument here. The following lemma contains the part of the proof where it is required to make analytic perturbations. 
\begin{lemma} \label{lemma_anosov2}
Let $G \in \E_{Q_0}$ and let $g = g(G)$ denote the corresponding Riemannian metric on $M_0$. Let $x \in M_0$, and suppose $u \in T_x M_0$ is such that $g(x)(u,u) = 1$ and $(x,u)$ is a periodic point of $\phi^t_G$ with minimal period $L$. Let $v,w \in T_x M_0$ such that
\begin{equation} \label{eq_anosovlemma2_1}
g(x) (u,v) = g(x) (u,w) = 0.
\end{equation}
Then there is $\bar{G} \in T_G \E_{Q_0}$ such that
\begin{equation} \label{eq_anosovlemma2_2}
D_G \Phi ((x,u), L, G)(\bar{G}) = (v; w)
\end{equation}
where $D_G$ denotes the derivative with respect to $G$.
\end{lemma}

Before proceeding to the proof of Lemma \ref{lemma_anosov2}, we first prove a preliminary lemma. Fix $(\theta, L, G) \in S M_0 \times \mathbb{R}_{> 0} \times \E_{Q_0}$, and let $\bar{G} \in T_G \E_{Q_0}$, so there is a real-analytic family $G_{\epsilon} \subset \E_{Q_0}$ with $G_0 \equiv G$ and $\left. \frac{d}{d \epsilon} \right|_{\epsilon = 0} G_{\epsilon} = \bar{G}$. Define the curve 
\begin{equation}
\gamma_{\epsilon} : [0,L + 1] \longrightarrow M_0
\end{equation}
by
\begin{equation}
\gamma_{\epsilon} (t) = \pi \circ \phi^t_{G_{\epsilon}} (\theta) \gamma (t) + \epsilon \bar{\gamma} (t) + O(\epsilon^2).
\end{equation}

\begin{lemma} \label{lemma_hvsplittingcharacterisation}
Let $D_G$ denote the derivative with respect to $G$, and let $\nabla_t = \nabla_{\gamma' (t)}$ denote the covariant time derivative with respect to the Levi-Civita connection of the metric $g=g(G)$. Then we have
\begin{equation}
D_G \Phi (\theta, L, G)(\bar{G}) = (\bar{\gamma}(L); \nabla_t \bar{\gamma}(L)).
\end{equation}
\end{lemma}

\begin{proof}
Write $\theta = (x,u)$ where $x \in T M_0$ and $u \in T_x M_0$, and let $V \in T_u T M_0$. Let $(U, \varphi)$ be a coordinate chart with $x \in U$ and write
\begin{equation} \label{eq_hvsplittingchart}
\varphi(x) = y, \quad D \varphi (u) = (y,v), \quad D^2 \varphi (V) = (y,v,w,z).
\end{equation}
Let $\Gamma_{ijk} (\cdot) = \Gamma_{ijk} (\cdot, G)$ denote the Christoffel symbols with respect to $g = g(G)$ and the chart $(U, \varphi)$, and let $\Gamma (y)$ denote the $(d+1)$-dimensional vector valued function with components
\begin{equation}
\left(\Gamma (y)(v,w) \right)_i = \sum_{j,k=0}^d \Gamma_{ijk}(y)v_jw_k.
\end{equation}
As pointed out in \cite{anosov1967geodesic, anosov1983generic} (see also \cite{klingenberg1976lectures}) we have $V=(Y; Z)$ where
\begin{equation} \label{eq_hvsplittingcoords1}
D \varphi (Y) = (y,w), \quad D \varphi (Z) = (y, z + \Gamma(y)(v,w)).
\end{equation}
Let us show that
\begin{equation} \label{eq_hvsplittingcoords2}
(\gamma (t), \gamma' (t), \bar{\gamma}(t), \bar{\gamma}' (t)) = (\bar{\gamma}(t); \nabla_t \bar{\gamma}(t)).
\end{equation}
If $(\gamma, \gamma', \bar{\gamma}, \bar{\gamma}') = (Y;Z)$ then \eqref{eq_hvsplittingchart} and the first equation of \eqref{eq_hvsplittingcoords1} imply that $D \varphi (Y) = (\gamma, \bar{\gamma})$,
and so $Y= \bar{\gamma}$. Moreover, \eqref{eq_hvsplittingchart} and the second equation of \eqref{eq_hvsplittingcoords1} give
\begin{equation}
D \varphi (Z) = (\gamma, \bar{\gamma}' + \Gamma (\gamma)(\gamma', \bar{\gamma}))  = (\gamma, \frac{d}{dt} \bar{\gamma} + \Gamma (\gamma) (\gamma', \bar{\gamma}))
\end{equation}
which implies that $Z = \nabla_t \bar{\gamma}$. This proves \eqref{eq_hvsplittingcoords2}. Using \eqref{eq_hvsplittingcoords2} we see that
\begin{align}
D_G \Phi (\theta, L, G)(\bar{G}) ={}& \left. \frac{d}{d \epsilon} \right|_{\epsilon = 0} \Phi (\theta, t, G_{\epsilon}) = \left. \frac{d}{d \epsilon} \right|_{\epsilon = 0} \phi^t_{G_{\epsilon}} (\theta) = \left. \frac{d}{d \epsilon} \right|_{\epsilon = 0} (\gamma_{\epsilon} (t), \gamma_{\epsilon}' (t)) \\
={}& (\gamma (t), \gamma' (t), \bar{\gamma}(t), \bar{\gamma}' (t))  = (\bar{\gamma}(t); \nabla_t \bar{\gamma}(t))
\end{align}
which completes the proof of the lemma.
\end{proof}

\begin{proof}[Proof of Lemma \ref{lemma_anosov2}]
Let $\gamma (t) = \pi \circ \phi^t_G (x,u)$ denote the closed geodesic corresponding to the orbit of $(x,u)$, and let $(y_0, y) = (y_0, y_1, \ldots, y_d)$ denote Fermi coordinates in a neighbourhood of $\gamma$ so that $\gamma (t) = (t, 0, \ldots, 0)$. Now if $v,w \in T_x M_0$ such that \eqref{eq_anosovlemma2_1} holds, then $v = (0, v_1, \ldots, v_d)$ and $w = (0, w_1, \ldots, w_d)$.

Let $Q \in \V$ such that $G(M_0)=M(Q) \equiv M$. Consider a real-analytic family $Q_{\epsilon} \subset \V$ with $Q_0 = Q$. Let $\gamma_{\epsilon}(t)$ denote the perturbed geodesic, and write
\begin{equation}
\gamma_{\epsilon}(t) = \gamma (t) + \epsilon \bar{\gamma}(t) + O(\epsilon^2).
\end{equation}
Let
\begin{equation} \label{eq_bumpymetricqbardef}
\bar{Q} = \left. \frac{d}{d \epsilon} \right|_{\epsilon = 0} Q_{\epsilon} \in T_Q \V.
\end{equation}
By Lemma \ref{lemma_hvsplittingcharacterisation}, if we can show that, by appropriate choice of real-analytic family $Q_{\epsilon}$, any vector $(\bar{\gamma}_1(L), \ldots, \bar{\gamma}_d(L), (\bar{\gamma}_1)'(L), \ldots, (\bar{\gamma}_d)'(L)) \in \mathbb{R}^{2d}$ can be obtained, then the lemma is proved.

Suppose we make a localised perturbation $Q \to Q + \epsilon \psi$, and denote by $\bar{g}$ the term of order $\epsilon$ in the Taylor expansion of the perturbed metric. Then $\bar{\gamma}(t)$ is the solution of the initial value problem
\begin{equation} \label{eq_bumpyivp1}
\frac{d^2}{dt^2} \bar{\gamma}_i (t) + \sum_{k=1}^d R_{i0k0} (t,0) \bar{\gamma}_k (t) + \bar{\Gamma}_{i00} (t,0) = 0, \quad \bar{\gamma}(0) = \frac{d}{dt} \bar{\gamma}(0) = 0
\end{equation}
where $R_{ijkl}$ is the Riemann curvature tensor, and 
\begin{equation} \label{eq_christoffelperturbation}
\bar{\Gamma}_{ijk} = \frac{1}{2} \left( \frac{\partial \bar{g}_{ij}}{\partial y_k} + \frac{\partial \bar{g}_{ik}}{\partial y_j} - \frac{\partial \bar{g}_{jk}}{\partial y_i} \right)
\end{equation}
are the Christoffel symbols corresponding to the perturbation $\bar{g}$ of the metric \cite{anosov1983generic}. 

Since $\gamma$ is a closed geodesic, it is easy to see that we can find a time $t$ for which $\kappa (\gamma (t), \gamma'(t)) \neq 0$ as a result of \eqref{eq_geodesicfloweom}. It follows that we must have $\kappa (\gamma (t), \gamma'(t)) \neq 0$ for all $t$ in some interval $[a,b]$. We assume that the perturbation $\psi$ is supported only when $y_0 \in (a,b)$. It follows from \eqref{eq_perturbationmatrix00} that $\tilde{C}_{00}(y_0,0) = - \kappa (\gamma (y_0), \gamma' (y_0)) \neq 0$ whenever $y_0 \in (a,b)$. Denote by $V_0 \subset \mathbb{R}^{d+1}$ a set consisting of points $(y_0,y)$ such that $y_0 \in (a,b)$ and $y$ is sufficiently small. For $(y_0,y) \in V_0$, define
\begin{equation} \label{eq_bumpymetricperturbation}
\tilde{\psi} (y_0, y) = \tilde{C}_{00}(y_0, y)^{-1} \sum_{j=1}^d y_j f_j (y_0)
\end{equation}
where $f_1, \ldots, f_d$ are functions of $y_0$ that we will choose later. We then define the $C^{\infty}$ perturbation $\psi$ of the hypersurface by $\psi (x) = \tilde{\psi} (h^{-1}(x))$ whenever $x \in h(V_0)$, and such that $\psi$ vanishes identically outside a small neighbourhood of $h(V_0)$ in $\mathbb{R}^{d+2}$. Combining \eqref{eq_metricpertcomps}, \eqref{eq_christoffelperturbation}, and \eqref{eq_bumpymetricperturbation}, we see that
\begin{equation}
\bar{\Gamma}_{000} (y_0,0) = \frac{1}{2} \frac{\bar{g}_{00}}{\partial y_0} (y_0,0) = 0, \quad \bar{\Gamma}_{i00} (y_0,0) = -f_i (y_0)
\end{equation}
for $i = 1, \ldots, d$. 

Since $\bar{g} \equiv 0$ along $\gamma$, we have $\bar{\gamma}_0 = 0$. Therefore we need only consider $\bar{\gamma}=(\bar{\gamma}_1, \ldots, \bar{\gamma}_d)$. Write
\begin{equation}
f(t) = (f_1(t), \ldots, f_d(t)), \quad R(t) = (R_{i0k0}(t,0))_{i,k=1, \ldots, d}.
\end{equation}
The initial value problem \eqref{eq_bumpyivp1} can be written as
\begin{equation}
\frac{d}{dt} \left(
\begin{matrix}
\bar{\gamma} (t) \\
\bar{\gamma}' (t)
\end{matrix}
\right) = \left(
\begin{matrix}
0 & I_d \\
-R(t) & 0
\end{matrix}
\right) \left(
\begin{matrix}
\bar{\gamma} (t) \\
\bar{\gamma}' (t)
\end{matrix}
\right) + \left(
\begin{matrix}
0 \\ f(t)
\end{matrix}
\right), \quad \left(
\begin{matrix}
\bar{\gamma}(0) \\
\bar{\gamma}' (0)
\end{matrix}
\right) = \left(
\begin{matrix} 0 \\ 0 \end{matrix}
\right).
\end{equation}
Let $U(t)$ denote the fundamental matrix solution of this initial value problem, so we have
\begin{equation} \label{eq_fundamentalmatrix}
\frac{d}{dt} U(t) = \left(
\begin{matrix}
0 & I_d \\
-R(t) & 0
\end{matrix}
\right) U(t), \quad U(0) = I_{2d}.
\end{equation}
Then $U(t)$ is a $2d \times 2d$ invertible matrix, and upon differentiating the identity $I_{2d} = U^{-1}(t) U(t)$ we see that its inverse satisfies
\begin{equation} \label{eq_fundamentalmatrixinverse}
\frac{d}{dt} U^{-1}(t) = -U^{-1}(t) \left(
\begin{matrix}
0 & I_d \\
-R(t) & 0
\end{matrix}
\right).
\end{equation}
Choose some $t_0 \in (a,b)$ such that $\gamma (t_0)$ is not a point of self-intersection of $\gamma$. Let $\delta (t - t_0)$ denote the Dirac delta function, a generalised function that can be thought of heuristically as taking the `value' $\infty$ at $t = t_0$ and 0 elsewhere. Suppose
\begin{equation}
f(t) = \alpha \delta (t-t_0) + \beta \delta' (t-t_0)
\end{equation}
where $\alpha, \beta \in \mathbb{R}^d$. Then the equation for variation of parameters together with properties of the Dirac delta function and equation \eqref{eq_fundamentalmatrixinverse} implies
\begin{align}
\left( \begin{matrix} \bar{\gamma} (L) \\ \bar{\gamma}' (L) \end{matrix} \right) ={}& U(L) \int_0^L U^{-1}(t) \left( \begin{matrix} 0 \\ f(t) \end{matrix} \right) dt = U(L) \int_{t_0}^L U^{-1}(t) \left( \begin{matrix} 0 \\ f(t) \end{matrix} \right) dt \\
={}& U(L) U^{-1}(t_0) \left( \begin{matrix} 0 \\ \alpha \end{matrix} \right) - U(L) \frac{d}{dt} U^{-1}(t_0) \left( \begin{matrix} 0 \\ \beta \end{matrix} \right) \\
={}& U(L)U^{-1}(t_0) \left( \begin{matrix} 0 \\ \alpha \end{matrix} \right) + U(L) U^{-1}(t_0) \left(
\begin{matrix} 0 & I_d \\ -R(t) & 0 \end{matrix}
\right) \left( \begin{matrix} 0 \\ \beta \end{matrix} \right) \\
={}& U(L) U^{-1}(t_0) \left( \begin{matrix} \beta \\ \alpha \end{matrix} \right).
\end{align}
Therefore, by appropriate choice of $\alpha, \beta$ we can obtain any vectors $\bar{\gamma}(L)$, $\bar{\gamma}'(L) \in \mathbb{R}^d$. However, $\delta$ is not a function, so instead choose some bump function $\varphi : \mathbb{R} \to \mathbb{R}$ such that
\begin{equation}
\textrm{supp}(\varphi) \subseteq (-1,1), \quad \int_{- \infty}^{\infty} \varphi (t) dt = 1.
\end{equation}
For $\epsilon_0 >0$ define $\varphi_{\epsilon_0}(t) = \frac{1}{\epsilon_0} \varphi \left( \frac{t}{\epsilon_0} \right)$. Then we have
\begin{equation}
\textrm{supp} (\varphi_{\epsilon_0}) \subseteq (-\epsilon_0, \epsilon_0), \quad \int_{- \infty}^{\infty} \varphi_{\epsilon_0} (t) dt = 1.
\end{equation}
Then $\varphi_{\epsilon_0} \to \delta$ as $\epsilon_0 \to 0$. For sufficiently small $\epsilon_0 >0$, let $f(t) = \alpha \varphi_{\epsilon_0} (t-t_0) + \beta \varphi_{\epsilon_0}'(t-t_0)$. Then we can still obtain any vectors by varying $\alpha, \beta$. Finally, approximate $Q + \epsilon \psi$ by a one-parameter real-analytic family $Q_{\epsilon}$ such that the term $\bar{Q}$ of order $\epsilon$ in the expansion of $Q_{\epsilon}$ (see \eqref{eq_bumpymetricqbardef}) is sufficiently close to $\psi$ along $\gamma$, along with its first derivative. Now, we can find a real-analytic family $G_{\epsilon} \subset \E_{Q_0}$ with the property that $G_{\epsilon} (M_0) = M(Q_{\epsilon})$ for all $\epsilon$. Let 
\begin{equation}
\bar{G} = \left. \frac{d}{d \epsilon} \right|_{\epsilon = 0} G_{\epsilon}.
\end{equation}
Then $\bar{G} \in T_G \E_{Q_0}$ satisfies \eqref{eq_anosovlemma2_2}.
\end{proof}

\section{Generic Properties of $k$-Jets of Poincar\'e Maps} \label{section_klingenbergtakens}

Let $Q \in \V$, and denote by $M=M(Q)$ the corresponding closed $C^{\omega}$-smooth hypersurface of $\mathbb{R}^{d+2}$ equipped with the Euclidean metric. Let $\gamma : [0,1] \to M$ denote a nonconstant geodesic segment, and denote by $l = (\gamma, \gamma')$ the corresponding orbit segment of the geodesic flow. Let $\Sigma_0, \Sigma_1$ denote transverse sections to $l$ in $TM$ at $l(0), l(1)$ respectively, restricted to the energy level of $l$. Let $P_{Q,1}: \Sigma_0 \to \Sigma_1$ denote the corresponding Poincar\'e map. 

We assume moreover that the normal curvature $\kappa (\gamma (t), \gamma'(t))$ is not identically zero, and $\gamma (0)$ is not a point of self-intersection of $\gamma$. It follows from \eqref{eq_geodesicfloweom} that we can find $0<a<b<c<1$ and a sufficiently small cylindrical neighbourhood $W$ of the curve segment $\gamma([0,c])$ such that:
\begin{itemize}
\item
$\kappa (\gamma (t), \gamma'(t)) \neq 0$ for all $t \in [a,b]$; 
\item
$\gamma ([0,1]) \cap W = \gamma ([0,c])$; and
\item
$W$ is contained in a neighbourhood of $\gamma$ in which Fermi coordinates are valid, and the Fermi coordinates define a chart in $W$.
\end{itemize}
Notice that the third point implies that $\gamma$ has no self-intersections in $W$. Let $k \in \mathbb{N}$ and $\J \subset J^k_s (d)$ be open, dense, and invariant.

\begin{theorem} \label{theorem_ktlocallysupportedperturbations}
There are $n_0 \in \mathbb{N}$ and $C^{\infty}$-smooth locally-supported functions $\psi_j : \mathbb{R}^{d+2} \to \mathbb{R}$ for $j=1, \ldots, n_0$ such that:
\begin{enumerate}
\item
For each $j$, we have $\gamma(0), \gamma(1) \notin \textrm{supp}(\psi_j)$;
\item
For each $j$, we have $\textrm{supp}(\psi_j) \cap M \subset W$;
\item
$\gamma$ is still a geodesic on the perturbed manifold
\begin{equation}
M_{\epsilon} = M(Q + \psi_{\epsilon})
\end{equation}
where
\begin{equation}
\psi_{\epsilon} = \epsilon_1 \psi_1 + \cdots + \epsilon_{n_0} \psi_{n_0}
\end{equation}
for all sufficiently small values of $\epsilon = (\epsilon_1, \ldots,\epsilon_{n_0})$, in the sense that $\phi^t_{M_{\epsilon}}(\gamma(0), \gamma'(0)) = (\gamma(t), \gamma'(t))$ where $\phi^t_{M_{\epsilon}}$ denotes the geodesic flow on $M_{\epsilon}$ with respect to the Euclidean metric; and
\item
There are arbitrarily small values of the parameters $\epsilon = (\epsilon_1, \ldots,\epsilon_{n_0})$ for which the $k$-jet $J^k_{l(0)} P_{Q+ \psi_{\epsilon},1}$ of the Poincar\'e map of $\gamma$ on the perturbed manifold $M_{\epsilon}$ lies in the set $\J$.
\end{enumerate}
\end{theorem}

This section is devoted to the proof of Theorem \ref{theorem_ktlocallysupportedperturbations}. Theorem \ref{theorem_ktlocal} follows immediately from Theorem \ref{theorem_ktlocallysupportedperturbations}, since we can approximate $Q+ \psi_{\epsilon}$ sufficiently well by a real-analytic family $Q_{\epsilon} = Q_{(\epsilon_1, \ldots, \epsilon_{n_0})}$ where $Q_0 \equiv Q$. Then we can find arbitrarily small values of the $n_0$ parameters $\epsilon$ for which the $k$-jet of the Poincar\'e map of the perturbed geodesic $\tilde{\gamma}$ on the manifold $M(Q_{\epsilon})$ lies in $\J$.

\subsection{Effect of a Perturbation on the $k$-Jet} \label{sec_effectpertkjet}

Let $(y_0,y)$ denote Fermi coordinates adapted to $\gamma$, and let $(v_0,v)$ denote the corresponding tangent coordinates, as described in Section \ref{sec_fermicoordinates}. Therefore \eqref{eq_geodesicinfermicoords} is the equation of $l$. For each $t \in [0,1]$ let
\begin{equation}
\Sigma (t) = \{ (y_0,y,v_0,v) : y_0 = t \}.
\end{equation}
Then $\Sigma (t)$ is a smooth family of transverse sections to $l$ in $TM$. Suppose now we make a perturbation as in \eqref{eq_perturbation} with $\psi$ supported near $\gamma$. Let
\begin{equation} \label{eq_poincaremap}
P_{Q+ \epsilon \psi, t} : \Sigma (0) \longrightarrow \Sigma (t)
\end{equation}
denote the Poincar\'e map (in the above Fermi coordinates) corresponding to the geodesic flow on $M(Q + \epsilon \psi)$ (i.e. $P_{Q,t}$ denotes the unperturbed Poincar\'e map from $\Sigma (0)$ to $\Sigma (t)$). Define the map
\begin{equation} \label{eq_perturbationfunction}
R_{Q, \epsilon \psi, t} : \Sigma (0) \longrightarrow \Sigma (0)
\end{equation}
by
\begin{equation} \label{eq_functionofthejetperturbation}
R_{Q, \epsilon \psi, t} = P^{-1}_{Q,t} \circ P_{Q + \epsilon \psi, t}.
\end{equation}
\begin{remark}
The maps $P_{Q + \epsilon \psi, t}$ and $R_{Q, \epsilon \psi, t}$ may not be defined on all of $\Sigma (0)$, so we restrict them to a neighbourhood of the point $l(0)$ in $\Sigma (0)$ where they are defined, and keep the notation as in \eqref{eq_poincaremap} and \eqref{eq_perturbationfunction} for convenience.
\end{remark}
Let $\bar{X}$ denote the vector field of the perturbation, and assume:
\begin{enumerate}[(i)]
\item
$J^{k-1}_{l(t)} \bar{X} = 0$ for all $t \in [0,1]$;
\item
$l(0), l(1) \notin \textrm{supp} \left( \bar{X} \right)$; and
\item
$\left. \bar{X} \right|_{\Sigma (t)}$ is tangent to $\Sigma (t)$ for all $t \in [0,1]$.
\end{enumerate}
Consider the nonautonomous vector field
\begin{equation} \label{eq_nonatuonomousvectorfield}
\bar{X}_t = P_{Q,t}^* \left( \bar{X} |_{\Sigma (t)} \right)
\end{equation}
defined as the pullback to $\Sigma (0)$ under the unperturbed Poincar\'e map $P_{Q,t}$ of the restriction of $\bar{X}$ to the transverse section $\Sigma (t)$. Let $\phi^t_{\epsilon \bar{X}_t}$ denote the time-$t$ shift along orbits of the flow of $\epsilon \bar{X}_t$. The following result relating the $k$-jet of this flow with the $k$-jet of $R_{Q, \epsilon \psi, t}$ is analogous to Proposition 2.1 in \cite{klingenberg1972generic}.

\begin{proposition} \label{proposition_perturbationfunctionkjet}
$J^k_{l(0)} R_{Q, \epsilon \psi, t} = J^k_{l(0)} \phi^t_{\epsilon \bar{X}_t} + O(\epsilon^2)$ for each $t \in [0,1]$.
\end{proposition}
For each $t \in [0,1]$ define
\begin{equation} \label{eq_transversesectiontildedef}
\tilde{\Sigma}(t) = \{ (y_0,y,v_0,v) : y_0 = t, v_0 = 1 \}.
\end{equation}
Clearly $\tilde{\Sigma}(t) \subset \Sigma (t)$. However, the manifold $\tilde{\Sigma} (0)$ is not necessarily invariant under $R_{Q, \epsilon \psi, t}$. The following lemma was proved in \cite{klingenberg1972generic} (Lemma 3.1 and Remark 2.3).
\begin{lemma} \label{lemma_ktremarks}
$P^{-1}_{Q,t}(\tilde{\Sigma}(t))$ and $\Sigma (0)$ are tangent to order $k$ at $l(0)$. Consequently, $\bar{X}_t$ is tangent to $\tilde{\Sigma}(0)$ with respect to $k$-jets, and so $\tilde{\Sigma}(0)$ and $R_{Q, \epsilon \psi, t} (\tilde{ \Sigma}(0))$ have a tangency of order $k$ at $l(0)$.
\end{lemma}

In what follows, we define a perturbation $\psi$ of the function $Q$ in terms of the Fermi coordinates $(y_0,y)$ near the curve $(t,0)$. From this we can define a function $\tilde{\psi} (x) = \psi (h^{-1} (y_0,y))$ for points $x$ on $M$ (where $h$ is defined by \eqref{eq_fermichartmap}), and extend it smoothly to a neighbourhood of these points, so that the perturbation is $C^{\infty}$ and locally supported. Define the perturbation space $\P^k$ as the space of $C^{\infty}$-smooth functions $\psi : \mathbb{R}^{d+2} \to \mathbb{R}$ such that $M(Q + \epsilon \psi)$ is a $C^{\infty}$-smooth closed hypersurface of $\mathbb{R}^{d+2}$ for all sufficiently smooth values of $\epsilon$, and such that, near $\gamma$, $\psi$ takes the form $\psi (y_0, y) = - \alpha (y_0) \beta (y)$ such that:
\begin{itemize}
\item
$\alpha$ is a real-valued function in $y_0$ with $\textrm{supp} (\alpha) \subset (a,b);$
\item
$\beta$ is a real-valued function in $y = (y_1, \ldots, y_d)$ with $J^k_{0} \beta = 0$ for each $t \in [0,1]$, and $\beta$ is supported sufficiently close to the origin so that $\textrm{supp}(\psi) \cap M \subset W.$
\end{itemize}

Notice that the restriction of the symplectic form $\omega$ to $\tilde{\Sigma}(t)$ is a symplectic form, and so it makes sense to discuss Hamiltonian functions and Hamiltonian vector fields on $\tilde{\Sigma}(t)$.

\begin{proposition} \label{proposition_kjetperturbationeffect}
Let $\psi \in \P^k$ where $\psi (y_0,y) = -\alpha (y_0) \beta (y)$. Then $J^k_{l(0)} R_{Q, \epsilon \psi, t}$ is equal (up to terms of order $\epsilon^2$) to the $k$-jet at $l(0)$ of the time-$t$ shift along orbits of the flow of the nonautonomous Hamiltonian function
\begin{equation}
\epsilon \left[ \tilde{H}_t \right] \circ P_{Q,t} : \tilde{\Sigma}(0) \longrightarrow \mathbb{R}
\end{equation}
where 
\begin{equation}
\left[ \tilde{H}_t \right] (y) = \kappa (\gamma (t), \gamma'(t)) \alpha (t)  \beta(y)
\end{equation}
is a one-parameter family (with parameter $t \in [0,1]$, but supported only when $t \in [a,b]$) of homogeneous polynomials of degree $k+1$ in $y = (y_1, \ldots, y_d)$ that is entirely determined by our choice of $\psi \in \P^k$ and the normal curvature along $\gamma$. Moreover, any one-parameter family of homogeneous polynomials of degree $k+1$ in $y = (y_1, \ldots, y_d)$ with $t$-support in $(a,b)$ occurs for some $\psi \in \P^k$.
\end{proposition}

\begin{proof}
Suppose we make the perturbation $Q \to Q + \epsilon \psi$ with $\psi \in \P^k$. It follows from \eqref{eq_metricpertcomps} that the Hamiltonian of the perturbation is 
\begin{equation}
\bar{H} (y_0,y,v_0,v) = \psi (y_0,y) \sum_{i,j=0}^d \tilde{C}_{ij} (y_0,y) v_i v_j.
\end{equation}
Let $\tilde{H}_t = \left. \bar{H} \right|_{\tilde{\Sigma} (t)}$ so that
\begin{equation}
\tilde{H}_t (y,v) = \psi (t,y) \left[ \tilde{C}_{00}(t,y) + 2 \sum_{j=1}^d \tilde{C}_{0j} (t,y) v_j + \sum_{i,j=1}^d \tilde{C}_{ij}(t,y) v_i v_j \right].
\end{equation}
Recall that $l(t) = (\gamma (t), \gamma' (t))$ where $\gamma (t) = (t, 0, \ldots, 0)$, and $ \gamma' (t) = (1, 0, \ldots, 0).$
It follows from this and the definition of $\P^k$ that 
\begin{equation}
J^k_{\gamma (t)} \psi = - \alpha (t) J^k_{0} \beta = 0, \quad J^{k+1}_{\gamma (t)} \psi = - \alpha (t) J^{k+1}_{0} \beta.
\end{equation}
Now, by \eqref{eq_perturbationmatrix00}, the $(k+1)$-jet of $\tilde{H}_t$ is
\begin{equation}
J^{k+1}_{l(t)} \tilde{H}_t = \kappa (\gamma (t), \gamma'(t)) \alpha (t) J^{k+1}_{0} \beta = J^{k+1}_{l(0)}\left[ \tilde{H}_t \right].
\end{equation}
We have $0 \neq \kappa (\gamma (t), \gamma' (t))$ for all $t \in [a,b]$ by assumption. Since along $\gamma$, $\alpha$ can be any smooth function in $y_0$ and $\beta$ can be any smooth function in $y = (y_1, \ldots, y_d)$ with vanishing $k$-jet, we can obtain any one-parameter family of homogeneous polynomials of degree $k+1$ in $y = (y_1, \ldots, y_d)$ with $t$-support in $(a,b)$ by varying our choice of $\psi \in \P^k$.

Notice that perturbations $\psi \in \P^k$ give rise to perturbative vector fields $\bar{X}$ satisfying assumptions (i)-(iii) of Proposition \ref{proposition_perturbationfunctionkjet}. Therefore, combining Proposition \ref{proposition_perturbationfunctionkjet} and Lemma \ref{lemma_ktremarks}, we find that
\begin{equation}
J^k_{l(0)} R_{Q, \epsilon \psi, t} = J^k_{l(0)} \phi^t_{\epsilon \tilde{X}_t} + O(\epsilon^2)
\end{equation}
where $\tilde{X}_t = P^*_{Q,t} \left( \left. \bar{X} \right|_{\tilde{\Sigma} (t)} \right)$. Since $P^*_{Q,t}$ is determined by the 1-jet of the symplectic map $P_{Q,t}$, and since $\tilde{H}_t$ is the Hamiltonian function with Hamiltonian vector field $\left. \bar{X} \right|_{\tilde{\Sigma}(t)}$, the $k$-jet of $\tilde{X}_t$ is the $k$-jet of the Hamiltonian vector field with Hamiltonian function $\left[ \tilde{H}_t \right] \circ P_{Q,t}$.
\end{proof}

\subsection{$k$-General Position of Families of Poincar\'e Maps}

Recall that the set $J^1_s (d)$ of 1-jets of symplectic automorphisms of $\mathbb{R}^{2d}$ that fix the origin is just the set $\textrm{Sp} (2d, \mathbb{R})$ of real $2d \times 2d$ symplectic matrices, that is, matrices $\sigma$ satisfying $\sigma^T \Omega \sigma = \Omega$ where
\begin{equation}
\Omega = \left(
\begin{array}{cc}
0 & I_{d} \\
-I_{d} & 0
\end{array} \right)
\end{equation}
is the standard symplectic matrix.

Let $\mathbb{R}_k [y,v]$ denote the set of real homogeneous polynomials of degree $k$ in $(y,v) = (y_1, \ldots, y_d, v_1, \ldots, v_d)$. This is a real vector space of dimension
\begin{equation} \label{eq_polynomialspacedimension}
N = \dim \mathbb{R}_k [y,v] = {2d - 1 + k \choose k}.
\end{equation}
\begin{definition}
A vector of matrices $(\sigma_1, \ldots, \sigma_N) \in \textrm{Sp} (2d, \mathbb{R})^N$ is $k$-\emph{general} if there are homogeneous polynomials $f_1, \ldots, f_N \in \mathbb{R}_k [y]$ such that $\{ f_1 \circ \sigma_1, \ldots, f_N \circ \sigma_N \}$ forms a basis of $\mathbb{R}_k [y,v]$. Write
\begin{equation}
G_k = \left\{ ( \sigma_1, \ldots, \sigma_N) \in \textrm{Sp} (2d, \mathbb{R})^N : ( \sigma_1, \ldots, \sigma_N) \textrm{ is $k$-general} \right\}.
\end{equation}
\end{definition}

\begin{proposition} \label{proposition_gkisdense}
$G_k$ is open and dense in $\textrm{Sp} (2d, \mathbb{R})^N$ for each $k \in \mathbb{N}$.
\end{proposition}
\begin{proof}
Let $F(y,v) = y_1^k$. It is shown in \cite{carballo2013jets} (Proposition 6) that the set
\begin{equation}
\left\{ (\sigma_1, \ldots, \sigma_N) \in \textrm{Sp} (2d, \mathbb{R})^N : \{ F \circ \sigma_1, \ldots, F \circ \sigma_N \} \textrm{ is a basis of } \mathbb{R}_k [y,v] \right\}
\end{equation}
is dense in $\textrm{Sp} (2d, \mathbb{R})^N$. This set is contained in $G_k$, so $G_k$ is dense. Since a sufficiently small perturbation of a basis is a basis, $G_k$ is open.
\end{proof}

The following basic linear algebra result is the key step in passing from the $C^{\infty}$ to the $C^{\omega}$ topology. 

\begin{lemma} \label{lemma_linalgbarbsmallbasis}
Let $N \in \mathbb{N}$ and let $V$ be an $N$-dimensional vector space. Let $u_1, \ldots, u_N$, $v_1, \ldots, v_N \in V$ and $\epsilon^*>0$ such that
\begin{equation} \label{eq_initialbasis}
u_1 + \epsilon^* v_1, \ldots, u_N + \epsilon^* v_N
\end{equation}
is a basis of $V$. Then
\begin{equation} \label{eq_arbsmallbasis}
u_1 + \epsilon v_1, \ldots, u_N + \epsilon v_N
\end{equation}
is a basis of $V$ for all but a finite number of $\epsilon \in [0, \epsilon^*]$. 
\end{lemma}

\begin{proof}
Consider the matrix $A(\epsilon) = \left[ u_1 + \epsilon v_1 | \cdots | u_N + \epsilon v_N \right]$ and the degree $N$ polynomial $f (\epsilon) = \det (A(\epsilon))$. Then the vectors \eqref{eq_arbsmallbasis} form a basis of $V$ if and only if $f(\epsilon) \neq 0$. Since \eqref{eq_initialbasis} is a basis of $V$, we have $f(\epsilon^*) \neq 0$. Therefore $f$ has at most $N$ zeros in $[0, \epsilon^*]$, and so there are at most $N$ values of $\epsilon \in [0, \epsilon^*]$ for which \eqref{eq_arbsmallbasis} is not a basis of $V$.
\end{proof}

\begin{definition}
Let $\left\{ \sigma_t \right\}_{t \in [0,1]}$ be a one-parameter family of symplectic automorphisms of $\mathbb{R}^{2d}$ that fix the origin. This family is $k$-\emph{general} if there are times $t_1, \ldots, t_N \in [0,1]$ such that
\begin{equation}
\left( J^1_0 \sigma_{t_1}, \ldots, J^1_0 \sigma_{t_N} \right) \in G_k.
\end{equation}
\end{definition}

\begin{proposition} \label{proposition_densekgenerality}
Let the assumptions and notation be as in the statement of Theorem \ref{theorem_ktlocallysupportedperturbations}. Denote by $(y_0,y)$ Fermi coordinates adapted to $\gamma$, and by $(v_0,v)$ the corresponding tangent coordinates. Let the transverse sections $\tilde{\Sigma}(t)$ be defined as in \eqref{eq_transversesectiontildedef}. Then we can find $\psi \in \P^1$ such that the family of linearised Poincar\'e maps
\begin{equation}
\left\{ D_{l(0)} P_{Q + \epsilon \psi,t} : T_{l(0)} \tilde{\Sigma}(0) \longrightarrow T_{l(t)} \tilde{\Sigma} (t) \right\}_{t \in [0,1]} 
\end{equation}
is $k$-general for arbitrarily small values of $\epsilon$ where the times $t_1, \ldots, t_N$ in the definition of $k$-generality satisfy
\begin{equation} \label{eq_kgentimesinab}
a < t_1 < \cdots < t_N < b.
\end{equation}
\end{proposition}

\begin{proof}
The easiest way to prove the proposition is to start with the analogous result for the $C^{\infty}$ case and use Lemma \ref{lemma_linalgbarbsmallbasis} to pass to the $C^{\omega}$ case. By Proposition \ref{proposition_gkisdense}, $k$-generality is an open and dense property of 1-jets. It was shown by Stojanov and Takens that we can find an arbitrarily $C^{\infty}$-small function $\tilde{\psi} \in \P^1$ such that the family of differentials 
\begin{equation}
\left\{ D_{l(0)} P_{Q + \tilde{\psi}, t} : T_{l(0)} \tilde{\Sigma}(0) \to T_{l(t)} \tilde{\Sigma}(t) \right\}_{t \in [0,1]}
\end{equation}
is $k$-general, where the times $t_1, \ldots, t_N$ in the definition of $k$-generality satisfy \eqref{eq_kgentimesinab} \cite{stojanov1993generic}. Choose some small $\epsilon^* >0$ and write
\begin{equation} \label{eq_insertionofsmallparameter}
\tilde{\psi} = \epsilon^* \psi.
\end{equation}
Then there are times $a < t_1 < \cdots < t_N < b$ and homogeneous polynomials $f_1, \ldots, f_N \in \mathbb{R}_k [y]$ such that $f_1 \circ \sigma_{1, \epsilon^*}, \ldots, f_N \circ \sigma_{N, \epsilon^*}$ is a basis of $\mathbb{R}_k [y, v]$, where $\sigma_{j, \epsilon} = D_{l(0)} P_{Q + \epsilon \psi, t_j}$ for $\epsilon \in [0, \epsilon^*]$. Momentarily ignoring the second order terms in $\epsilon^*$, we may assume, shrinking $\epsilon^*$ and redefining \eqref{eq_insertionofsmallparameter} if necessary, that
\begin{equation}
f_1 \circ \sigma_1 + \epsilon^* \bar{f}_1,\ldots, f_N \circ \sigma_N + \epsilon^* \bar{f}_N
\end{equation}
is a basis of $\mathbb{R}_k [y, v]$, where $\sigma_j = \sigma_{j,0}$ and 
\begin{equation}
\bar{f}_j = \left. \frac{d}{d \epsilon} \right|_{\epsilon = 0} f_j \circ \sigma_{j, \epsilon} \in \mathbb{R}_k [y, v].
\end{equation}
By Lemma \ref{lemma_linalgbarbsmallbasis}, there are at most finitely many $\epsilon \in (0, \epsilon_*)$ for which
\begin{equation}
f_1 \circ \sigma_1 + \epsilon \bar{f}_1,\ldots, f_N \circ \sigma_N + \epsilon \bar{f}_N
\end{equation}
is not a basis of $\mathbb{R}_k [y, v]$. Then, since a sufficiently small perturbation of a basis is a basis, 
\begin{equation}
\left\{ D_{l(0)} P_{Q + \epsilon \psi, t} : T_{l(0)} \tilde{\Sigma}(0) \longrightarrow T_{l(t)} \tilde{\Sigma}(t) \right\}_{t \in [0,1]}
\end{equation}
is $k$-general for arbitrarily small values of $\epsilon$, where the $k$-generality times satisfy \eqref{eq_kgentimesinab}.
\end{proof}

\subsection{Proof of Theorem \ref{theorem_ktlocallysupportedperturbations}}

Let the assumptions and notation be as in the statement of Theorem \ref{theorem_ktlocallysupportedperturbations}, with $(y_0,y)$ Fermi coordinates, $(v_0,v)$ the corresponding tangent coordinates, and the transverse sections $\tilde{\Sigma}(t)$ as defined in \eqref{eq_transversesectiontildedef}. By Proposition \ref{proposition_densekgenerality} we may assume without loss of generality (in the context of Theorem \ref{theorem_ktlocallysupportedperturbations}) that the family of Poincar\'e maps $\{ D_{l(0)} P_{Q,t} \}_{t \in [0,1]}$ is $(m+1)$-general for each $m=1, \ldots, k$, where the times in the definition of $m$-generality satisfy \eqref{eq_kgentimesinab}. Recall the definition of the space $\P^m$ of perturbations from Section \ref{sec_effectpertkjet}, and consider the map
\begin{align}
S_m : \P^m & \longrightarrow J^m_s(d) \\
\tilde{\psi} & \longmapsto J^m_{l(0)} R_{Q,\tilde{\psi},1}.
\end{align}
Denote by
\begin{equation}
\pi_m : J^m_s(d) \longrightarrow J^{m-1}_s (d)
\end{equation}
the projection by truncation. Then the kernel of $\pi_m$ is the set of $\sigma \in J^m_s(d)$ for which $\pi_m(\sigma)$ is equal to the $(m-1)$-jet of the identity map on $\mathbb{R}^{2d}$. It is clear that the image of $S_m$ is contained in the kernel of $\pi_m$.

\begin{lemma} \label{lemma_imageopeninkernel}
For each $m=1, \ldots, k$ we can find $\psi_{m,1}, \ldots, \psi_{m, N_m} \in \P^m$, where
\begin{equation}
N_m = \dim \mathbb{R}_{m+1} [y,v]
\end{equation}
such that the map
\begin{align}
 : \mathbb{R}^{N_m} & \longrightarrow \ker(\pi_m) \\
(\epsilon_{1}, \ldots, \epsilon_{N_m}) & \longmapsto S_m (\epsilon_{1} \psi_{m,1} + \cdots + \epsilon_{N_m} \psi_{m,N_m})
\end{align}
is open in a neighbourhood of $0 \in \mathbb{R}^{N_m}$.
\end{lemma}

\begin{proof}
Fix some $m \in \{ 1, \ldots, k \}$. Since $\{ D_{l(0)} P_{Q,t} \}_{t \in [0,1]}$ is $(m+1)$-general with times in $(a,b)$, there are $f_1, \ldots, f_{N} \in \mathbb{R}_{m+1}[y]$ and times $a < t_{1} < \cdots < t_{N} < b$ such that
\begin{equation} \label{eq_mplusonegenbasis}
f_1 \circ D_{l(0)} P_{Q,t_{1}}, \ldots, f_{N} \circ D_{l(0)} P_{Q,t_{N}}
\end{equation} 
is a basis of $\mathbb{R}_{m+1} [y,v]$ where
\begin{equation}
N = \dim \mathbb{R}_{m+1} [y,v].
\end{equation}
As in the proof of Lemma \ref{lemma_anosov2}, we will first use delta functions to perform computations, and then approximate the delta functions by $C^{\infty}$-smooth locally-supported functions in $\P^m$ to complete the proof of the Lemma.

Since the normal curvature along the geodesic segment $\gamma ([a,b])$ is nonvanishing, and since the times $t_n$ are in $(a,b)$, we have $\kappa (\gamma(t_n), \gamma'(t_n)) \neq 0$. For each $n =1, \ldots, N$ let
\begin{equation}
\tilde{\psi}_n (y_0,y) = - \kappa (\gamma (y_0), \gamma'(y_0))^{-1} \delta (y_0 - t_n) f_n (y).
\end{equation}
Write
\begin{equation}
\tilde{\psi}_{\epsilon} = \sum_{n=1}^N \epsilon_n \tilde{\psi}_n
\end{equation}
where $\epsilon = (\epsilon_1, \ldots, \epsilon_{N})$ takes values in a neighbourhood of $0 \in \mathbb{R}^{N}$. By Proposition \ref{proposition_kjetperturbationeffect}, the $m$-jet at $l(0)$ of $R_{Q, \tilde{\psi}_{\epsilon}, 1}$ is equal (up to terms of order $\epsilon^2$) to the $m$-jet at $l(0)$ of the time-1 shift of the flow of the nonautonomous Hamiltonian function
\begin{equation}
\sum_{n=1}^N \epsilon_n \delta(t - t_n) f_n \circ P_{Q,t}.
\end{equation}
Denote by 
\begin{equation}
Y = \sum_{n=1}^N Y_n
\end{equation}
the corresponding Hamiltonian vector field, with the summands $Y_n$ being given by
\begin{equation}
Y_n = \epsilon_n \delta (t-t_n) \Omega \nabla \left(f_n \circ P_{Q,t} \right)
\end{equation}
where $\Omega$ is the standard $2d \times 2d$ symplectic matrix. Let $\phi^t_Y$ denote the flow of the vector field $Y$. Let $(y,v) \in \tilde{\Sigma}(0)$ and write $(y(t),v(t)) = \phi^t_Y(y,v)$. We have
\begin{align}
\phi^1_Y (y,v) ={}& (y,v) + \sum_{n=1}^N \epsilon_n \int_0^1 Y_n (y(t), v(t); t) dt \\
={}& (y,v) + \sum_{n=1}^N \epsilon_n \int_0^1 \delta (t - t_n) \Omega \nabla (f_n \circ P_{Q,t})(y(t),v(t)) dt \\
={}& (y,v) + \\
&+ \sum_{n=1}^N \epsilon_n \int_0^1 \delta (t - t_n) \Omega \left[ \nabla f_n (P_{Q,t}(y(t),v(t))) \right]^{T} D_{(y(t),v(t))} P_{Q,t} dt \\
={}& (y,v) + \sum_{n=1}^N \epsilon_n \Omega \left[ \nabla f_n (P_{Q,t_n}(y(t_n),v(t_n))) \right]^{T} D_{(y(t_n),v(t_n))} P_{Q,t_n}.
\end{align}
Let $I$ denote the identity map on $\mathbb{R}^{2d}$. Since $J^m_{l(0)} f_n = 0$, the $(m-1)$-jet of $\nabla f_n$ at $l(0)$ is zero. Moreover, since $l(0)$ is an equilibrium of $Y$ we find that
\begin{align}
J^m_{l(0)} R_{Q, \tilde{\psi}_{\epsilon},1} ={}& J^m_{l(0)} \phi^1_Y + O(\epsilon^2) \\
={}& J^1_{l(0)} I + \sum_{n=1}^N \epsilon_n \Omega \left( J^m_{l(0)} \nabla f_n \right) \circ D_{l(0)} P_{Q, t_n} + O(\epsilon^2). \label{eq_rofperturbationcomputation}
\end{align}
Since any homogeneous polynomial of degree $(m+1)$ in $(y,v)$ can be obtained as a linear combination of the polynomials \eqref{eq_mplusonegenbasis}, from \eqref{eq_rofperturbationcomputation} we can obtain the 1-jet of the identity plus any homogeneous $m$-jet of a symplectic automorphism of $\mathbb{R}^{2d}$ by varying $\epsilon = (\epsilon_1, \ldots, \epsilon_N)$. Therefore the map
\begin{equation}
(\epsilon_1, \ldots, \epsilon_N) \longmapsto S_m (\epsilon_1 \tilde{\psi_1} + \cdots + \epsilon_N \tilde{\psi}_N)
\end{equation}
is open.

Now, choose a $C^{\infty}$ bump function $\varphi : \mathbb{R} \to \mathbb{R}$ such that 
\begin{equation}
\textrm{supp} (\varphi) \subset [-1,1], \quad \int_{- \infty}^{\infty} \varphi (t) dt = 1.
\end{equation}
For small $\epsilon_0 >0$ let $\varphi_{\epsilon_0} (t) = \frac{1}{\epsilon_0} \varphi \left( \frac{t}{\epsilon_0} \right)$. Then
\begin{equation}
\textrm{supp} (\varphi_{\epsilon_0}) \subset [-\epsilon_0,\epsilon_0], \quad \int_{- \infty}^{\infty} \varphi_{\epsilon_0} (t) dt = 1.
\end{equation}
It follows that $\varphi_{\epsilon_0} \to \delta$ as $\epsilon_0 \to 0$. Fix some sufficiently small $\epsilon_0 >0$ and let
\begin{equation}
\psi_n (y_0,y) = - \kappa (\gamma (y_0), \gamma' (y_0))^{-1} \varphi_{\epsilon_0} (y_0 - t_n) f_n (y).
\end{equation}
Since $\epsilon_0$ is sufficiently small, we have
\begin{equation}
\textrm{supp} (\psi_{n_1}) \cap \textrm{supp} (\psi_{n_2}) = \emptyset
\end{equation}
if $n_1 \neq n_2$, and moreover
\begin{equation}
\textrm{supp} (\psi_n) \subset (a,b).
\end{equation}
Therefore $\psi_1, \ldots, \psi_N \in \P^m$, and since a sufficiently small perturbation of a basis is a basis, we obtain that the map
\begin{equation}
(\epsilon_1, \ldots, \epsilon_N) \longmapsto S_m \left( \epsilon_1 \psi_1 + \cdots + \epsilon_N \psi_N \right)
\end{equation}
is open in a neighbourhood of the origin, as required.
\end{proof}

\begin{proof}[Proof of Theorem \ref{theorem_ktlocallysupportedperturbations}]
Let $\X^k$ denote the space of $C^{\infty}$-smooth functions $\psi : \mathbb{R}^{d+2} \to \mathbb{R}$ of the form $\psi = \psi_1 + \cdots + \psi_N$ for some $N \in \mathbb{N}$ where
\begin{equation}
\psi_1, \ldots, \psi_N \in \P^1 \cup \cdots \cup \P^k,
\end{equation}
and define the map
\begin{align}
S : \X^k & \longrightarrow J^k_s(d) \\
\psi & \longmapsto J^k_{l(0)} R_{Q,\psi,1}.
\end{align}
Consider the functions $\{ \psi_{m,n} \}_{m=1, \ldots, k}^{n = 1, \ldots, N_m}$ found in Lemma \ref{lemma_imageopeninkernel}. Relabel these functions as $\psi_1, \ldots, \psi_N$ where
\begin{equation}
N = \sum_{m=1}^{k} N_m.
\end{equation}
By Lemma \ref{lemma_imageopeninkernel}, the map $: \mathbb{R}^N \rightarrow J^k_s(d)$ defined by
\begin{equation}
(\epsilon_1, \ldots, \epsilon_N) \longmapsto S(\epsilon_1 \psi_1 + \cdots + \epsilon_N \psi_N)
\end{equation}
is open in a neighbourhood of $0 \in \mathbb{R}^N$. Since $\J$ is open and dense, there exist arbitrarily small values of $\epsilon = (\epsilon_1, \ldots, \epsilon_N)$ for which $J^k_{l(0)} P_{Q + \psi_{\epsilon},1} \in \J$ where $\psi_{\epsilon} = \epsilon_1 \psi_1 + \cdots + \epsilon_N \psi_N$.
\end{proof}

\section{Transversality and the Kupka-Smale Theorem} \label{section_kupkasmale}

Recall that a submanifold $N$ of a symplectic manifold $(TM, \omega)$ is called \emph{Lagrangian} if the restriction of $\omega$ to $TN$ vanishes identically, and the dimension of $N$ is maximal with respect to this property (i.e. $\dim N = \frac{1}{2} \dim TM$). A proof of the following lemma can be found in Appendix A of \cite{contreras2002genericity}.
\begin{lemma} \label{lemma_lagrangiantangent}
If a Lagrangian submanifold $N$ is contained in an energy level $H^{-1}(E)$ of a Hamiltonian function $H$, then the Hamiltonian vector field associated with $H$ is tangent to $N$.
\end{lemma}

The following local perturbation result allows us to make a heteroclinic intersection transverse in a neighbourhood of a heteroclinic point using locally-supported perturbations. In practice, we will make these locally-supported perturbations and then approximate by a real-analytic family to ensure we remain in the space $\V$. This lemma is an adaptation of Lemma 2.6 in \cite{contreras2002genericity}.

\begin{lemma} \label{lemma_localtransversality}
Let $Q \in \V$ and $M = M(Q)$, and suppose $\gamma, \eta$ are hyperbolic closed geodesics on $M$ (where $\gamma$ is allowed to equal $\eta$). Choose $t_0$ such that
\begin{equation} \label{eq_notflatassumption}
\kappa (\gamma (t_0), \gamma' (t_0)) \neq 0
\end{equation}
and let $\U$ be a sufficiently small open neighbourhood of $(\gamma (t_0), \gamma' (t_0))$ in $TM$. Suppose $\theta \in W^u (\gamma) \cap W^s (\eta) \cap \U$ is such that the projection $\pi : TM \to M$ is a diffeomorphism in a neighbourhood of $\theta$ in $W^u (\gamma)$. Assume moreover that $\theta \in T^1 M$ has unit length. Take any sufficiently small neighbourhoods $U_1 \subset U_2 \subset U_3$ of $\theta$ in $W^u (\gamma)$ such that $\overline{U}_3 \subset \U$. Then there is a $C^{\infty}$-smooth function $\psi : \mathbb{R}^{d+2} \to \mathbb{R}$ such that
\begin{equation}
\textrm{\normalfont supp} (\psi) \cap M \subseteq \pi (U_3)
\end{equation}
and, with $M_{\epsilon} = M(Q + \epsilon \psi)$, for all sufficiently small values of $\epsilon$, the connected component of $W^u (\gamma) \cap U_1$ containing $\theta$ is transverse to $W^s (\eta)$.
\end{lemma}

\begin{proof}
Since $\U$ is siffuciently small, and since $\theta \in \U$, we may pass to Fermi coordinates $(y_0, y) = (y_0, y_1, \ldots, y_d)$ adapted to $\gamma$ (see Section \ref{sec_fermicoordinates}). Let $(v_0, v) = (v_0, v_1, \ldots, v_d)$ denote the corresponding tangent coordinates. Then we know that the geodesic flow is the flow of the Hamiltonian function
\begin{equation}
H(y_0, y, v_0, v) = \frac{1}{2} g (y_0, y) ((v_0,v), (v_0,v)) = \frac{1}{2} \sum_{i,j=0}^d g_{ij} (y_0, y) v_i v_j.
\end{equation}
Since $\pi |_{W^u (\gamma)}$ is a diffeomorphism in a neighbourhood of $\theta$, the connected component $W$ of $W^u (\gamma) \cap U_3$ containing $\theta$ is 
\begin{equation}
W = \textrm{graph} (v^*) = \{ ((y_0,y), v^* (y_0,y)) : (y_0, y) \in \pi (W) \}
\end{equation}
where $v^* = (v^*_0, v^*_1, \ldots, v^*_d)$. A result of Robinson implies that there is a $C^{\infty}$ smooth potential $V$ (that does not necessarily correspond to a geodesic flow) such that 
\begin{equation}
\textrm{supp} (V) \subseteq U_2
\end{equation}
and if we write
\begin{equation}
\tilde{H}_{\epsilon} = H + \epsilon V
\end{equation}
then the perturbed unstable manifold $\tilde{W}^u_{\epsilon} (\gamma)$ corresponding to the Hamiltonian flow of $\tilde{H}_{\epsilon}$ is transverse to the unperturbed stable manifold $W^s (\eta)$ in the neighbourhood $U_1$ of $\theta$ for arbitrarily small values\footnote{In the proof of Claim a, Theorem 3 of \cite{robinson1970generic}, $c$ is the small parameter. The claim is not proved explicitly in that paper, but it is equivalent to the proof of Claim a of Theorem 1 of \cite{robinson1970generic}.} of the parameter \cite{robinson1970generic} (see also \cite{1078-0947_2008_2_551}). Since $W^s (\eta)$ is unchanged up to the support of the perturbation, the new stable and unstable manifolds are transverse.

Let $W_{\epsilon}$ denote the connected component of $\tilde{W}^u_{\epsilon} (\gamma) \cap U_3$ containing $\theta$. If $\epsilon$ is small enough then $\pi |_{W_{\epsilon}}$ is still a diffeomorphism, so we can write
\begin{equation}
W_{\epsilon} = \textrm{graph} (v^* + \epsilon \bar{v}^*) = \{ ((y_0,y), v^*(y_0,y) + \epsilon 
\bar{v}^* (y_0,y)) : (y_0,y) \in \pi (W_{\epsilon}) \}
\end{equation}
where $\bar{v}^*$ may depend on $\epsilon$. Now make the perturbation $Q \to Q + \epsilon \psi$ and consider the perturbed metric $g_{\epsilon} = g + \epsilon \bar{g} + O(\epsilon^2)$ where $\bar{g}_{ij}$ is defined by \eqref{eq_metricpertcomps}. Write
\begin{equation}
\tilde{\kappa} ((y_0,y),(v_0,v)) = \sum_{i,j=0}^d \tilde{C}_{ij} (y_0,y) v_i v_j.
\end{equation}
Due to \eqref{eq_perturbationmatrix00} and \eqref{eq_notflatassumption} we have
\begin{equation}
\tilde{\kappa} ((t_0,0),(1,0)) = \tilde{C}_{00} (t_0,0) = \kappa (\gamma (t_0), \gamma' (t_0)) \neq 0.
\end{equation}
Since $\U$ is a sufficiently small neighbourhood of $(\gamma (t_0), \gamma' (t_0))$ we have
\begin{equation} \label{eq_notflatassumption2}
\tilde{\kappa} ((y_0,y),(v_0,v)) \neq 0
\end{equation}
for all $((y_0,y),(v_0,v)) \in \U$.

The Hamiltonian of the perturbed geodesic flow is
\begin{equation}
H_{\epsilon} (y_0,y,v_0,v) = \frac{1}{2} g(y_0,y) ((v_0,v), (v_0,v)) + \frac{1}{2} \epsilon \bar{g} (y_0,y) ((v_0,v), (v_0,v)) + O(\epsilon^2).
\end{equation}
On the manifold $W_{\epsilon}$ we have 
\begin{equation}
\hspace{-14em} g_{\epsilon} (y_0,y) ( v^*(y_0,y) + \epsilon \bar{v}^* (y_0,y), v^*(y_0,y) + \epsilon \bar{v}^* (y_0,y)) = 
\end{equation}
\vspace*{-2.5em}
\begin{align}
\begin{split}
={}& g (y_0,y)(v^*(y_0,y) + \epsilon \bar{v}^* (y_0,y), v^*(y_0,y) + \epsilon \bar{v}^* (y_0,y)) + \\
 &+ \epsilon \bar{g} (v^* (y_0,y), v^* (y_0,y)) + O(\epsilon^2)
\end{split} \\
\begin{split}
={}& g (y_0,y) (v^* (y_0,y), v^* (y_0,y)) + \\
& + \epsilon \left[ 2 \psi (y_0,y) \tilde{\kappa} ((y_0,y), v^* (y_0,y)) + 2 g (y_0,y) ( v^* (y_0,y), \bar{v}^* (y_0,y)) \right] + O(\epsilon^2).
\end{split}
\end{align}
Notice that the term of order $\epsilon$ vanishes if we let
\begin{equation}
\psi (y_0,y) = 
\begin{cases}
- \frac{g (y_0,y)(v^*(y_0,y), \bar{v}^* (y_0,y))}{\tilde{\kappa} ((y_0,y), v^* (y_0,y))} & \textrm{ if } (y_0,y) \in \pi(U_2)\\
0 & \textrm{ if } (y_0,y) \notin \pi(U_3)
\end{cases}
\end{equation}
with a smooth transition on $\pi(U_3) \setminus \pi(U_2)$. This is well-defined due to \eqref{eq_notflatassumption2}. Consider the truncated perturbed Hamiltonian
\begin{equation}
H_{\epsilon}^* = H + \epsilon \bar{H}
\end{equation}
where we have removed the terms of order $\epsilon^2$. On the manifold $W_{\epsilon}$ we have
\begin{equation}
H_{\epsilon}^* ((y_0,y), v^* (y_0,y) + \epsilon \bar{v}^* (y_0,y)) = H ((y_0,y), v^* (y_0,y)) = \frac{1}{2}
\end{equation}
since
\begin{equation}
W = \textrm{graph} (v^*) \subset T^1 M.
\end{equation}
It follows that
\begin{equation}
W_{\epsilon} \subset \left( H_{\epsilon}^* \right)^{-1} \left(\frac{1}{2} \right).
\end{equation}
It is well-known that the stable and unstable manifolds of a hyperbolic periodic orbit of a Hamiltonian flow are Lagrangian submanifolds. Therefore, by Lemma \ref{lemma_lagrangiantangent}, the Hamiltonian vector field of $H_{\epsilon}^*$ is tangent to $W_{\epsilon}$. It follows that $W_{\epsilon}$ is a piece of the unstable manifold of $\gamma$ with respect to the flow of $H_{\epsilon}^*$. Moreover, $W_{\epsilon}$ is transverse to $W^s(\eta)$. Since $H_{\epsilon}$ is $O(\epsilon^2)$-close to $H_{\epsilon}^*$, the unstable manifold of $\gamma$ is transverse to the stable manifold of $\eta$ with respect to $H_{\epsilon}$ in the neighbourhood $U_1$ of $\theta$.
\end{proof}

The following result is also useful. A proof can be found in \cite{paternain2012geodesic} (Proposition 2.11).

\begin{proposition} \label{proposition_paternain}
Let $\phi^t : TM \to TM$ denote the geodesic flow, $\pi : TM \to M$ the canonical projection, and let $\theta \in TM$. Suppose $L \subset T_{\theta} TM$ is a Lagrangian subspace. Then the set
\begin{equation}
\left\{ t \in \mathbb{R} : D_{\theta} \phi^t (L) \cap \ker \left(D_{\phi^t (\theta)} \pi \right) \neq \{ 0 \} \right\}
\end{equation}
is discrete.
\end{proposition}

\begin{remark} \label{remark_transversality}
In Lemma \ref{lemma_localtransversality}, it is assumed that the restriction of the projection $\pi : TM \to M$ to $W^u (\gamma)$ is a diffeomorphism in a neighbourhood of a point $\theta \in W^u (\gamma)$. Proposition \ref{proposition_paternain} combined with the inverse function theorem tells us that this assumption fails to be true on an at most countable set of points on the orbit of $\theta$. Therefore if we start at a point $\theta \in W^u (\gamma)$ where the projection is not a diffeomorphism, then we can find arbitrarily small $t$ so that it is a diffeomorphism at $\phi^t (\theta)$, and so we can apply the lemma. Since transversality of stable and unstable manifolds is a property of orbits and not just points, if the perturbed invariant manifolds are transverse in a neighbourhood of $\phi^t (\theta)$, then they are transverse in a neighbourhood of $\theta$.
\end{remark}

\begin{theorem} \label{theorem_transversalityofconnections}
Let $Q \in \mathcal{V}$ and $M = M(Q)$, and suppose $\gamma, \eta$ are hyperbolic closed geodesics with a heteroclinic (or homoclinic in the case where $\gamma = \eta$) connection. Then there is $\tilde{Q}$ arbitrarily close to $Q$ in $\mathcal{V}$ such that the hyperbolic closed geodesics $\gamma, \eta$ persist on $M(\tilde{Q})$, and the connection between them is transverse.
\end{theorem}

\begin{proof}
Since $\gamma : \mathbb{T} \to M$ is a closed geodesic, there exists $t_0 \in \mathbb{T}$ such that $\kappa (\gamma (t_0), \gamma' (t_0)) \neq 0$, as a result of \eqref{eq_geodesicfloweom}. Let $\U$ be a sufficiently small neighbourhood of $(\gamma (t_0), \gamma' (t_0))$ in $TM$, as in Lemma \ref{lemma_localtransversality}.

Suppose $W^u (\gamma) \cap W^s (\eta) \neq \emptyset$. Recall that a fundamental domain for the geodesic flow $\phi^t$ on $W^u (\gamma)$ is a subset of $W^u (\gamma)$ that intersects every orbit of $\phi^t |_{W^u (\gamma)}$ exactly once. We can find arbitrarily small fundamental domains arbitrarily close to $\gamma$. Let $K \subset W^u (\gamma) \cap \U$ be a sufficiently small compact set, sufficiently close to $\gamma$, such that $K$ contains a fundamental domain of $\phi^t |_{W^u (\gamma)}$.

Let $\theta_1 \in K$, and suppose $\pi |_{W^u(\gamma)}$ is a diffeomorphism in a neighbourhood $K_1$ of $\theta_1$, where $K_1$ is sufficiently small (i.e. $K_1$ plays the role of the set $U_1$ in Lemma \ref{lemma_localtransversality}). Then, by Lemma \ref{lemma_localtransversality}, there is a locally supported function $\psi_1$ such that for arbitrarily small values of $\epsilon$, the connected component of $W^u(\gamma) \cap K_1$ containing $\theta_1$ is transverse to $W^s(\eta)$ on the manifold $M(Q + \epsilon \psi_1)$.

If $\theta_1$ does not have the property that $\pi |_{W^u(\gamma)}$ is a diffeomorphism in a neighbourhood of $\theta_1$, then we can find some small time $t$ such that $\phi^t (\theta_1) \in \U$ does have that property, by Proposition \ref{proposition_paternain}. Then let $\bar{K}_1$ denote a sufficiently small neighbourhood of $\phi^t (\theta_1)$ where we can apply Lemma \ref{lemma_localtransversality} to find the appropriate perturbation $\psi_1$, localised near $\phi^t (\theta_1)$. Let $K_1 = \phi^{-t}(\bar{K}_1)$. Then for arbitrarily small values of $\epsilon$ we get transversality of $W^s(\eta)$ and $W^u(\gamma)$ in the set $K_1 \cap M(Q+ \epsilon \psi_1)$ (see Remark \ref{remark_transversality}).

Since $K$ is compact, there is $n \in \mathbb{N}$ such that for $j = 1, \ldots, n$ we have a neighbourhood $K_j$ of $\theta_j$ in $W^u(\gamma)$ and locally supported functions $\psi_j$ such that $W^s(\eta)$ and $W^u(\gamma)$ are transverse in $K_j$ on the manifold $M(Q + \epsilon \psi_j)$, and
\begin{equation}
K \subset \bigcup_{j=1}^n K_j.
\end{equation}
We first make a sufficiently small perturbation $Q \to Q + \epsilon_1 \psi_1$ to get transversality on $K_1$. Since the property of $W^s(\eta)$ and $W^u(\gamma)$ being transverse in $K_1$ is open, we may then find sufficiently small $\epsilon_2$ such that if we make the perturbation $Q+ \epsilon_1 \psi_1 \to Q+ \epsilon_1 \psi_1+ \epsilon_2 \psi_2$, we obtain transversality of $W^s(\eta)$ and $W^u(\gamma)$ in $K_2$ without destroying the transversality in $K_1$. Repeating this process $n$ times, each time taking care not to destroy transversality in the previous neighbourhood $K_{j}$, we find that there are arbitrarily small values of the parameters $\epsilon = (\epsilon_1, \ldots, \epsilon_n)$ such that if
\begin{equation}
\tilde{Q}_{\epsilon} = Q + \sum_{j=1}^n \epsilon_j \psi_j,
\end{equation}
then the stable manifold of $\eta$ is transverse to the unstable manifold of $\gamma$ on the manifold $M(\tilde{Q}_{\epsilon})$. We can now approximate $\tilde{Q}_{\epsilon}$ arbitrarily well by an $n$-parameter real-analytic family $Q_{\epsilon}$. If the approximation is sufficiently good in the $C^4$-topology, then $\eta, \gamma$ are still hyperbolic closed geodesics on the manifold $M(Q_{\epsilon})$, and $W^s(\eta), W^u(\gamma)$ are still transverse for arbitrarily small values of $\epsilon$. Therefore we can take $\epsilon$ small enough to find $\tilde{Q}$ arbitrarily close to $Q$ in $\V$ for which we get the desired transversality property. 
\end{proof}

\begin{proof}[Proof of Theorem \ref{theorem_kupkasmale}]
Let $N \in \mathbb{N}$ and let $\H(N)$ denote the set of $Q \in \V$ for which the $k$-jet of the Poincar\'e map of every closed geodesic of length at most $N$ in $M(Q)$ lies in $\J$. Intersect $\J$ with the set of $k$-jets of symplectic automorphisms of $\mathbb{R}^{2d}$ whose linearisation at the origin does not have an eigenvalue equal to 1. Then $\J$ is still open, dense, and invariant. On such a manifold $M(Q)$ there can be only finitely many periodic orbits of length at most $N$. Therefore $\H(N)$ is open since $\J$ is. It follows that there is an open and dense set $\B (N) \subset \V$ such that for all $Q \in \B (N)$, every closed geodesic of length at most $N$ on $M(Q)$ is nondegenerate. Then for each $Q$ in $\B(N)$, the manifold $M(Q)$ admits only finitely many closed geodesics of length at most $N$.

Now, let $Q_0 \in \V$. Arbitrarily close to $Q_0$ we can find $Q_1 \in \B(N)$ since $\B(N)$ is dense. Then there are only finitely many closed geodesics of length at most $N$ on $M(Q_1)$. Therefore we can apply Theorem \ref{theorem_ktlocal} to each of these closed geodesics to find some $Q_2$ arbitrarily close to $Q_1$ such that the $k$-jet of the Poincar\'e map of every closed geodesic of length at most $N$ in $M(Q_2)$ is in $\J$. Therefore $Q_2 \in \H^k(N)$. Since $Q_2$ is arbitrarily close to $Q_0$, this proves that $\H(N)$ is dense. This proves part \eqref{item_kupkasmale1} of Theorem \ref{theorem_kupkasmale}.

Let $\K(N)$ denote the set of $Q \in \H(N)$ such that if $\gamma, \eta$ are hyperbolic closed geodesics on $M(Q)$ of length at most $N$, then $W^s(\eta), W^u(\gamma)$ are transverse. Since transversality is an open property, $\K(N)$ is open. It remains to prove that $\K(N)$ is dense.

Let $Q \in \H(N)$, and let $\gamma, \eta$ be hyperbolic closed geodesics on $M(Q)$ of length at most $N$. If there is no heteroclinic connection (or homoclinic if $\gamma = \eta$) between $\gamma$ and $\eta$, or if any such connection is transverse, there is nothing to prove. Otherwise, there is a non-transverse intersection. We may then apply Theorem \ref{theorem_transversalityofconnections} to make this intersection transverse.

Repeating this process for each of the finitely many pairs of hyperbolic closed geodesics of length at most $N$ completes the proof of density of $\K(N)$. It follows that
\begin{equation}
\K = \bigcap_{N \in \mathbb{N}} \K (N)
\end{equation}
is the residual set we are looking for.
\end{proof}

\section{Generic Existence of Hyperbolic Sets on Surfaces} \label{section_3dgenericityofhyperbolicsets}

In this section it is shown that the geodesic flow with respect to the Euclidean metric on real-analytic, closed, and strictly convex surfaces in $\mathbb{R}^3$ generically (i.e. for a $C^{\omega}$ open and dense set of such surfaces) has a transverse homoclinic orbit as a result of Theorems \ref{theorem_kupkasmale} and \ref{theorem_homoclinicnearelliptic}, a theorem of Mather \cite{mather1981invariant}, and an argument of Knieper and Weiss \cite{knieper2002c}, thus proving Theorem \ref{theorem_3dgenerichyperbolicset}. Since a surface $M =M(Q)$ having a transverse homoclinic orbit for its geodesic flow is a $C^4$-open property of $Q \in \V^c$, it is required to show only that it is $C^{\omega}$-dense.

Theorem \ref{theorem_homoclinicnearelliptic} implies that any closed strictly convex analytic surface in $\mathbb{R}^3$ with a nonhyperbolic closed geodesic $C^{\omega}$-generically has a transverse homoclinic to a hyperbolic closed geodesic. Moreover, Theorem \ref{theorem_kupkasmale} implies that, on $C^{\omega}$-generic surfaces, all homoclinic and heteroclinic connections are transverse. Therefore in what follows we consider only the case where every closed geodesic is hyperbolic and every homoclinic or heteroclinic connection is transverse (the ``Kupka-Smale condition'').

Let $Q \in \V^c$, $M = M(Q)$ such that the above assumptions hold, and recall that $\phi^t : TM \to TM$ denotes the geodesic flow with respect to the Euclidean metric on $M$. A variational construction due to Birkhoff implies the existence of a simple closed geodesic $\gamma : \mathbb{T} \to M$ which is referred to as the ``minimax'' geodesic \cite{birkhoff1927dynamical}. Let $A = \mathbb{T} \times [0, \pi]$. Notice that the simple closed curve $\gamma$ divides $M$ into two hemispheres. Since $M$ is strictly convex and $\gamma$ is the minimax geodesic, a theorem of Birkhoff (Section VI, 10 of \cite{birkhoff1927dynamical}) implies that there is a well-defined analytic function $\tau : \textrm{Int}( A) \to \mathbb{R}$ defined as follows. If $(\varphi, y) \in \textrm{Int} (A)$ and $x=\gamma (\varphi)$, then there is a uniquely defined unit tangent vector $u \in T^1_{x} M$ pointing into the northern hemisphere and making an angle $y$ with $\gamma' (\varphi)$. Then $\tau (\varphi, y)$ is the minimum value of $t >0$ for which $(\bar{x}, \bar{u}) = \phi^t (x,u)$ consists of a point $\bar{x}$ on the curve $\gamma$ and a unit tangent vector $\bar{u} \in T^1_{\bar{x}}M$ pointing again into the northern hemisphere. The function $\tau$ is then extended continuously to $\partial A$.

This implies the existence of a global Poincar\'e map $P:A \to A$ defined as follows. If $(\varphi, y) \in A$ and $(x,u)$ are as above, then let $(\bar{x}, \bar{u}) = \phi^{\tau (\varphi, y)} (x,u)$. Then there is $\bar{\varphi} \in \mathbb{T}$ such that $\bar{x} = \gamma (\bar{\varphi})$. Since $\bar{u}$ is pointing into the northern hemisphere, the angle $\bar{y}$ between $\bar{u}$ and $\gamma' (\bar{\varphi})$ is in $[0, \pi]$. Then $P(\varphi, y) = (\bar{\varphi}, \bar{y})$. Moreover $P$ is analytic on $\textrm{Int}(A)$ and continuous on $\partial A$.

Since $M$ is topologically a 2-sphere, the theorem of the three closed geodesics implies that $M$ has at least three geometrically distinct simple closed geodesics \cite{ballmann1978satz} (see also \cite{klingenberg1976lectures}). Let $\eta : \mathbb{T} \to M$ denote one of these simple closed geodesics, geometrically distinct from $\gamma$. This geodesic is hyperbolic by assumption, and therefore implies the existence of a hyperbolic fixed point $z \in \textrm{Int}(A)$ for (an iterate of) the global Poincar\'e map $P$. The idea of the proof is to apply the following theorem of Mather \cite{mather1981invariant}.

\begin{theorem}[Mather] \label{theorem_mather}
Let $U \subseteq \mathbb{S}^2$ be an open set and $f: U \to f(U) \subseteq \mathbb{S}^2$ a homeomorphism. Suppose every fixed point of $f$ is sectorial periodic or Moser stable, and suppose $f$ has no fixed connections. Let $O$ be a sectorial fixed point of $f$ and let $b_1, b_2$ be any two stable or unstable branches of $O$. Then the closures of $b_1$ and $b_2$ coincide.
\end{theorem}

Since the map must be defined on an open set, notice that we can extend $P$ continuously to an open annulus $A'$ containing $A$, for example by reflecting the dynamics of $P|_{\textrm{Int}(A)}$ across $\partial A$.

The precise definitions of \emph{sectorial periodic} and \emph{Moser stable} fixed points can be found in \cite{mather1981invariant}, but sectorial periodic points can be thought of as a topological alternative to hyperbolic periodic points with one contracting and one expanding eigendirection. Indeed, Mather points out in \cite{mather1981invariant} that hyperbolic fixed points of maps in dimension 2 with one expanding and one contracting eigendirection are sectorial periodic. Since all closed geodesics on $M$ are hyperbolic, so too are all fixed points of $P$. Moreover, $P$ is a symplectic map so the characteristic eigenvalues of hyperbolic fixed points are reciprocal pairs. Therefore every fixed point of $P$ in $\textrm{Int}(A)$ is sectorial periodic.

A \emph{fixed connection} is a $P$-invariant arc $\xi$ with fixed endpoints \cite{mather1981invariant}. The existence of such an arc for $P$ would imply the existence of a non-transverse homoclinic or heteroclinic connection, which is impossible due to the Kupka-Smale condition that we have assumed. It follows that $P$ satisfies the assumptions of Mather's theorem, and so the closure of all 4 stable and unstable branches of the hyperbolic fixed point $z$ coincide. Therefore the following result of Knieper and Weiss \cite{knieper2002c} applies.

\begin{proposition}[Knieper and Weiss] \label{proposition_knieperweiss}
Let $A$ be a surface to which the Jordan curve theorem can be applied, and let $f: A \to A$ be a diffeomorphism with a hyperbolic fixed point $z$. Suppose the closure of a branch of $W^s(z)$ coincides with the closure of a branch of $W^u(z)$. Then the two branches have a topological crossing.
\end{proposition}

The proposition implies that the stable and unstable manifolds of $z$ intersect. By the Kupka-Smale condition, this intersection is transverse. The transverse homoclinic intersection of $W^s(z), W^u (z)$ for $P$ implies the transverse homoclinic intersection of $W^s(\eta), W^u (\eta)$ for the geodesic flow, completing the proof of Theorem \ref{theorem_3dgenerichyperbolicset}.

\bibliographystyle{abbrv}
\bibliography{geodesic_flow_refs} 

\begin{thebibliography}{10}

\bibitem{abraham1970bumpy}
R.~Abraham.
\newblock Bumpy metrics.
\newblock In {\em Global Analysis, Proc. Sympos. Pure Math}, volume~14, pages
  1--3, 1970.

\bibitem{anosov1967geodesic}
D.~V. Anosov.
\newblock Geodesic flows on closed riemannian manifolds of negative curvature.
\newblock {\em Trudy Matematicheskogo Instituta Imeni VA Steklova}, 90:3--210,
  1967.

\bibitem{anosov1983generic}
D.~V. Anosov.
\newblock On generic properties of closed geodesics.
\newblock {\em Mathematics of the USSR-Izvestiya}, 21(1):1, 1983.

\bibitem{arnaud1992type}
M.-C. Arnaud.
\newblock Type des points fixes des diff{\'e}omorphismes symplectiques de
  $\mathbb{T}^n \times \mathbb{R}^n$.
\newblock {\em M{\'e}m. Soc. Math. France (NS) No}, 48:63, 1992.

\bibitem{ballmann1978satz}
W.~Ballmann.
\newblock Der satz von {L}usternik und {S}chnirelmann.
\newblock {\em Mathematische Shriften}, 102:1--25, 1978.

\bibitem{birkhoff1917dynamical}
G.~D. Birkhoff.
\newblock Dynamical systems with two degrees of freedom.
\newblock {\em Proceedings of the National Academy of Sciences of the United
  States of America}, 3(4):314, 1917.

\bibitem{birkhoff1927dynamical}
G.~D. Birkhoff.
\newblock {\em Dynamical systems}, volume~9.
\newblock American Mathematical Soc., 1927.

\bibitem{broer1986differentiable}
H.~Broer and F.~Tangerman.
\newblock From a differentiable to a real analytic perturbation theory,
  applications to the {K}upka {S}male theorems.
\newblock {\em Ergodic Theory and Dynamical Systems}, 6(3):345--362, 1986.

\bibitem{carballo2013jets}
C.~Carballo and J.~Miranda.
\newblock Jets of closed orbits of {M}an\'e's generic {H}amiltonian flows.
\newblock {\em Bulletin of the Brazilian Mathematical Society, New Series},
  44(2):219--232, 2013.

\bibitem{cheng2004existence}
C.-Q. Cheng and J.~Yan.
\newblock Existence of diffusion orbits in a priori unstable hamiltonian
  systems.
\newblock {\em Journal of Differential Geometry}, 67(3):457--517, 2004.

\bibitem{clarke2019arnold}
A.~Clarke and D.~Turaev.
\newblock Arnold diffusion in multi-dimensional convex billiards.
\newblock {\em arXiv preprint arXiv:1906.07778}, 2019.

\bibitem{colding2011course}
T.~H. Colding and W.~P. Minicozzi.
\newblock {\em A course in minimal surfaces}, volume 121.
\newblock American Mathematical Soc., 2011.

\bibitem{contreras2010geodesic}
G.~Contreras.
\newblock Geodesic flows with positive topological entropy, twist maps and
  hyperbolicity.
\newblock {\em Annals of mathematics}, pages 761--808, 2010.

\bibitem{contreras2002genericity}
G.~Contreras-Barandiar\'an and G.~P. Paternain.
\newblock Genericity of geodesic flows with positive topological entropy on
  $\mathbb{S}^2$.
\newblock {\em Journal of Differential Geometry}, 61(1):1--49, 2002.

\bibitem{croke1988area}
C.~B. Croke.
\newblock Area and the length of the shortest closed geodesic.
\newblock {\em Journal of Differential Geometry}, 27(1):1--21, 1988.

\bibitem{delshams2006orbits}
A.~Delshams, R.~De~La~Llave, and T.~M. Seara.
\newblock Orbits of unbounded energy in quasi-periodic perturbations of
  geodesic flows.
\newblock {\em Advances in Mathematics}, 202(1):64--188, 2006.

\bibitem{donnay1995transverse}
V.~J. Donnay.
\newblock Transverse homoclinic connections for geodesic flows.
\newblock In {\em Hamiltonian dynamical systems}, pages 115--125. Springer,
  1995.

\bibitem{gelfreich2017arnold}
V.~Gelfreich and D.~Turaev.
\newblock Arnold diffusion in a priori chaotic symplectic maps.
\newblock {\em Communications in Mathematical Physics}, 353(2):507--547, 2017.

\bibitem{gidea2017perturbations}
M.~Gidea and R.~de~la Llave.
\newblock Perturbations of geodesic flows by recurrent dynamics.
\newblock {\em Journal of the European Mathematical Society}, 19(3):905--956,
  2017.

\bibitem{gonchenko2007homoclinic}
S.~Gonchenko, D.~Turaev, and L.~Shilnikov.
\newblock Homoclinic tangencies of arbitrarily high orders in conservative and
  dissipative two-dimensional maps.
\newblock {\em Nonlinearity}, 20(2):241, 2007.

\bibitem{hofer2002pseudoholomorphic}
H.~Hofer, K.~Wysocki, and E.~Zehnder.
\newblock Pseudoholomorphic curves and dynamics in three dimensions.
\newblock In {\em Handbook of dynamical systems}, volume~1, pages 1129--1188.
  Elsevier, 2002.

\bibitem{hofer2003finite}
H.~Hofer, K.~Wysocki, and E.~Zehnder.
\newblock Finite energy foliations of tight three-spheres and hamiltonian
  dynamics.
\newblock {\em Annals of Mathematics}, pages 125--255, 2003.

\bibitem{jacobi1884vorlesungen}
C.~G.~J. Jacobi.
\newblock {\em Vorlesungen {\"u}ber dynamik}.
\newblock G. Reimer, Berlin, 1884.

\bibitem{katok1995introduction}
A.~Katok and B.~Hasselblatt.
\newblock {\em Introduction to the modern theory of dynamical systems},
  volume~54.
\newblock Cambridge university press, 1995.

\bibitem{klingenberg1976lectures}
W.~Klingenberg.
\newblock {\em Lectures on closed geodesics}.
\newblock Mathematisches Institut der Universitat Bonn, 1976.

\bibitem{klingenberg1972generic}
W.~Klingenberg and F.~Takens.
\newblock Generic properties of geodesic flows.
\newblock {\em Mathematische Annalen}, 197(4):323--334, 1972.

\bibitem{knieper2002c}
G.~Knieper and H.~Weiss.
\newblock {$C^{\infty}$} genericity of positive topological entropy for
  geodesic flows on $\mathbb{S}^2$.
\newblock {\em Journal of Differential Geometry}, 62(1):127--141, 2002.

\bibitem{le1991proprietes}
P.~Le~Calvez.
\newblock {\em Propri{\'e}t{\'e}s dynamiques des diff{\'e}omorphismes de
  l'anneau et du tore}.
\newblock Soci{\'e}t{\'e} math{\'e}matique de France, 1991.

\bibitem{mane1983oseledec}
R.~Ma{\~n}{\'e}.
\newblock Oseledec's theorem from the generic viewpoint.
\newblock In {\em Proceedings of the international Congress of Mathematicians},
  volume~1, page~2. Warsaw, 1983.

\bibitem{mather1981invariant}
J.~Mather.
\newblock Invariant subsets of area-preserving homeomorphisms of surfaces.
\newblock {\em Mathematical Analysis and Applications}, 7B:531--561, 1981.

\bibitem{moser1980geometry}
J.~Moser.
\newblock Geometry of quadrics and spectral theory.
\newblock In {\em The Chern Symposium 1979}, pages 147--188. Springer, 1980.

\bibitem{moser1980various}
J.~Moser.
\newblock Various aspects of integrable hamiltonian systems.
\newblock In {\em Dynamical systems}, pages 137--195. Springer, 1980.

\bibitem{1078-0947_2008_2_551}
E.~R. Oliveira.
\newblock Generic properties of lagrangians on surfaces: The kupka-smale
  theorem.
\newblock {\em Discrete \& Continuous Dynamical Systems}, 21(2):551--569, 2008.

\bibitem{paternain2012geodesic}
G.~P. Paternain.
\newblock {\em Geodesic flows}, volume 180.
\newblock Springer Science \& Business Media, 2012.

\bibitem{petroll1996existenz}
D.~Petroll.
\newblock {\em Existenz und Transversalit{\"a}t von homoklinen und heteroklinen
  Orbits beim geod{\"a}tischen Flu{\ss}}.
\newblock PhD thesis, University of Freiburg, 1996.

\bibitem{poincare1899methodes}
H.~Poincar\'e.
\newblock {\em Les m\'ethodes nouvelles de la m\'ecanique celeste}.
\newblock Springer, 1899.

\bibitem{rifford2012generic}
L.~Rifford and R.~O. Ruggiero.
\newblock Generic properties of closed orbits of hamiltonian flows from
  {M}a\~n\'e's viewpoint.
\newblock {\em International Mathematics Research Notices},
  2012(22):5246--5265, 2012.

\bibitem{robinson1970generic}
R.~C. Robinson.
\newblock Generic properties of conservative systems.
\newblock {\em American Journal of Mathematics}, 92(3):562--603, 1970.

\bibitem{shilnikov1965case}
L.~Shilnikov.
\newblock A case of the existence of a countable number of periodic orbits.
\newblock In {\em Sov. Math. Dokl}, volume~6, pages 163--166, 1965.

\bibitem{smale1967differentiable}
S.~Smale.
\newblock Differentiable dynamical systems.
\newblock {\em Bulletin of the American mathematical Society}, 73(6):747--817,
  1967.

\bibitem{stojanov1990bumpy}
L.~Stojanov.
\newblock A bumpy metric theorem and the poisson relation for generic strictly
  convex domains.
\newblock {\em Mathematische Annalen}, 287(1):675--696, 1990.

\bibitem{stojanov1993generic}
L.~Stojanov and F.~Takens.
\newblock Generic properties of closed geodesics on smooth hypersurfaces.
\newblock {\em Mathematische Annalen}, 296(1):385--402, 1993.

\bibitem{tabachnikov2002ellipsoids}
S.~L. Tabachnikov.
\newblock Ellipsoids, complete integrability and hyperbolic geometry.
\newblock {\em Moscow Mathematical Journal}, 2(1):183--196, 2002.

\bibitem{zehnder1973homoclinic}
E.~Zehnder.
\newblock Homoclinic points near elliptic fixed points.
\newblock {\em Communications on Pure and Applied Mathematics}, 26(2):131--182,
  1973.

\end{thebibliography}

\end{document}